\documentclass{amsart}
\usepackage{amsmath,amsthm,amssymb,amscd,tikz-cd,comment}
\theoremstyle{plain}

\usepackage{fullpage}

\usepackage{mathrsfs} 

\usepackage{wrapfig}

\usepackage[normalem]{ulem}
\usepackage{enumitem}
\usepackage{hyperref}
\usepackage{tabularx}
\newcolumntype{s}{>{\centering\hsize=.2\hsize \raggedright}X} 
\newcolumntype{b}{>{\centering\hsize=\hsize\raggedright\arraybackslash}X}

\newcolumntype{C}{>{\centering\arraybackslash}X} 

\usepackage{float}
\restylefloat{table}
\usepackage[justification=centering]{caption}
\usepackage{float}
\usepackage{wrapfig}
\restylefloat{figure}

\newcommand\robout{\bgroup\markoverwith {\textcolor{blue}{\rule[0.5ex]{2pt}{0.4pt}}}\ULon} 

\newcommand\johnout{\bgroup\markoverwith {\textcolor{magenta}{\rule[0.5ex]{2pt}{0.4pt}}}\ULon} 

\newtheorem{thm}{Theorem}[section]
\newtheorem{prop}[thm]{Proposition}
\newtheorem{lemma}[thm]{Lemma}
\newtheorem{cor}[thm]{Corollary}
\newtheorem{conj}[thm]{Conjecture}
\newtheorem{question}[thm]{Question}

\theoremstyle{definition}

\newtheorem{remark}[thm]{Remark}

\newtheorem{heuristic}[thm]{Heuristic}

\newcommand{\F}{{\mathbb F}}

\renewcommand{\L}{{\mathcal L}}

\newcommand{\ve}{\varepsilon}

\newcommand{\Z}{{\mathbb Z}}
\newcommand{\Q}{{\mathbb Q}}

\newcommand{\C}{{\mathbb C}}
\newcommand{\Fp}{{\mathbb F}_p}
\newcommand{\Fpx}{{\mathbb F}_p^\times}

\newcommand{\Fpbar}{\overline{{\mathbb F}}_p}
\newcommand{\Zpbar}{\overline{{\Z}}_p}
\newcommand{\Qpbar}{\overline{{\Q}}_p}

\newcommand{\Qbar}{\overline{{\Q}}}

\renewcommand{\P}{{\mathbb P}}

\newcommand{\Zp}{\Z_p}
\newcommand{\Zpx}{\Z_p^\times}

\newcommand{\Qp}{\Q_p}

\newcommand{\Spec}{\mathrm{Spec}}

\newcommand{\ds}{\displaystyle}
\newcommand{\rhobar}{\overline{\rho}}
\newcommand{\rbar}{\overline{r}}

\newcommand{\langx}{\langle x \rangle}
\renewcommand{\d}{~\!d}

\newcommand{\fk}{\frac{k-2}{2}}
\newcommand{\dz}{~\!dz}
\newcommand{\dmuf}{~\!d \mu_f(x)}

\DeclareMathOperator{\Spf}{Spf}

\DeclareMathOperator{\ord}{ord}

\DeclareMathOperator{\Fil}{Fil}

\DeclareMathOperator{\Sym}{Sym}

\DeclareMathOperator{\GL}{GL}
\DeclareMathOperator{\SL}{SL}

\DeclareMathOperator{\Gal}{Gal}

\DeclareMathOperator{\TS}{TS}
\DeclareMathOperator{\tr}{tr}

\DeclareMathOperator{\ind}{ind}
\DeclareMathOperator{\new}{new}
\newcommand{\pnew}{p\text{-new}}
\DeclareMathOperator{\unr}{unr}

\title{New phenomena arising from $\L$-invariants of modular forms}
\author{John Bergdall and Robert Pollack}

\address{John Bergdall\\Department of Mathematical Sciences\\University of Arkansas\\ 850 W Dickson Street \\Fayetteville, AR 72701\\USA}
\email{bergdall@uark.edu}
\urladdr{http://bergdall.github.io}

\address{Robert Pollack\\Department of Mathematics  \\  University of Arizona \\ 617 N Santa Rita Ave \\ Tucson, AZ 85721 \\USA}
\email{rpollack@arizona.edu}
\urladdr{http://www.math.arizona.edu/people/rpollack}



\date{\today}

\begin{document}

\begin{abstract}

This article explains how to practically compute $\L$-invariants of $p$-new eigenforms using $p$-adic $L$-series and exceptional zero phenomena. As proof of the utility, we compiled a data set consisting of over 150,000 $\L$-invariants. We analyze qualitative and quantitative features found in the data. This includes conjecturing a statistical law for the distribution of the valuations of $\L$-invariants in a fixed level as the weights of eigenforms approach infinity. One novel point of our investigation is that the algorithm is sensitive to compiling data for fixed Galois representations modulo $p$. Therefore, we explain new perspectives on $\L$-invariants that are related to Galois representations. We propose understanding the structures in our data through the lens of deformation rings and moduli stacks of Galois representations.

\end{abstract}

\maketitle

\setcounter{tocdepth}{1}
\tableofcontents

\section{Introduction}

This article explores a $p$-adic distribution phenomenon in the theory of elliptic modular forms. The reported research exists alongside wider statistical problems in arithmetic geometry and the arithmetic of modular forms. So, this introduction will describe established theories in its first two subsections, after which we bring the article's main character --- the $\L$-invariants --- into focus.

Ultimately, we achieve two goals in this article. First, we describe a practical method to compute $\L$-invariants. It uses their presence in the exceptional zero conjecture for $p$-adic $L$-functions, which was first formulated by Mazur, Tate, and Teitelbaum in the 1980's \cite{MTT}. The computational method itself is more recent, originally developed by the second author roughly 15 years ago, in order to perform {\em ad hoc} computations related to Buzzard's work on the slope problem for modular eigenforms \cite{Buzzard-SlopeQuestions}. Those computations have also been referenced twice in published work by the first author \cite{Bergdall-bounds,BLL}, creating a heightened need for the method's publication. 

Second, we formulate a conjecture on the $p$-adic distribution of $\L$-invariants. We support the conjecture with numerical data and heuristics.  The data sets we gathered are fairly large, numbering in the tens of thousands of eigenforms. When compared with analogous $a_p$-data for eigenforms, we find significant evidence of a yet-to-be formulated theory of ``$p$-adic distributions for automorphic points on generic fibers of deformation rings''. That is quite a phrase, of course. The interested reader may find further inspiration and discussion in the first few pages of  Buzzard and Gee's survey paper on slopes of modular forms \cite{BG-Survey}. In this paper, we will remain as concrete as possible, focused on slopes of $a_p$ and slopes of $\L$-invariants.

\subsection{Archimedean distributions of Hecke eigenvalues}
Fix a positive integer $M$ and let $\Gamma_0 = \Gamma_0(M)$. Suppose $f$ is an eigenform of level $\Gamma_0$ and weight $k$ with $q$-expansion $q + a_2(f)q^2 + a_3(f)q^3 + \dotsb$. A law governs the magnitude of $q$-series coefficients:\ if $p$ is a prime number and $p \nmid M$, then
$$
|a_p(f)| \leq 2p^{\frac{k-1}{2}}.
$$
This was proven by Deligne in 1970's, as a consequence of his work on the Weil conjectures (see \cite[No.\ 5]{Del-Bourbaki}). It is known as the Ramanujan--Petersson bound.

Fascinating statistical questions arise after normalizing by the Ramanujan--Petersson bound. The Sato--Tate conjecture for elliptic modular forms, proven in works of Harris and Taylor with Clozel, Shepherd-Barron, and Barnet-Lamb and Geraghty \cite{ClozelHarrisTaylor,HarrisShepherdBarronTaylor,BarnetLambGeraghtyHarrisTaylor-CalabaiYaus}, predicts the {\em real} distribution of 
\begin{equation}\label{eqn:ram-norm}
a_p(f)p^{\frac{1-k}{2}} \in [-2,2],
\end{equation}
for a fixed eigenform $f$ and $p$ ranging over the primes. Its importance is based on connections, independently made by Sato and Tate, to distributions of point-counts for the reductions of rational elliptic curves over finite fields. (See \cite[p.\ 106-107]{Tate-AlgCycles}.) Its proof is also a landmark in modern number theory, for it was driven by potential modularity techniques invented by Taylor and his collaborators, generalizing the modularity theorems of Wiles, Taylor and Wiles, and of Breuil, Conrad, Diamond, and Taylor \cite{Wiles-FLT, TaylorWiles, BCDT-FLT}.

The {\em vertical} Sato--Tate conjecture predicts the distribution of $a_p(f)p^{\frac{1-k}{2}}$ in a different aspect. Namely, one fixes the prime $p$ (still co-prime to $M$) and allows $k$ to grow to infinity. (In each weight $k$, the finite number of eigenforms are ordered in an arbitrary way.) The distribution of the normalized statistic was established by Serre and, at least for level one eigenforms, by Conrey, Duke, and Farmer \cite{Serre-Distribution,CDF}. (See also Sarnak's earlier article \cite{Sarnak-Stats}.)

Today, the Sato--Tate conjecture and its vertical analogue are understood as parts of wider conceptual apparatuses. For instance, both equidistribution statements are with respect to an {\em a priori} defined measure related to a linear group, either the Haar measure on the Sato--Tate group or the Plancherel measure for $\mathrm{PGL}_2(\Qp)$ in the vertical case. In addition, both versions have also been generalized beyond elliptic modular forms. Sutherland's Sato--Tate-based contribution to the Arizona Winter School is an excellent introduction to analogues of the classical conjecture \cite{Sutherland-SatoTate}, and we refer to articles of Shin and Dalal for discussion on generalizations of the vertical version \cite{Shin-AutomorphicPlancherel, Dalal-Equidistribution}.

\subsection{Non-Archimedean distributions  of Hecke eigenvalues}\label{subsec:nonarch-dist}

We now focus on $p$-adic properties, rather than real ones. Let $v_p$ be a $p$-adic valuation on $\Qbar$. The {\em slope} of an eigenform $f$ of level $\Gamma_0$ is $v_p(a_p(f))$. Predicting slopes as the weight varies is known as the {\em slope problem}. It was the focus of work by Buzzard and Gouv\^ea in the early 2000's  and the two authors' more recent ``ghost conjecture'' \cite{Buzzard-SlopeQuestions,Gouvea,BP-Ghost,BP-Ghost2}. A major breakthrough on the slope problem was recently made by Liu, Truong, Xiao, and Zhao \cite{LTXZ-Ghost,LTXZ-Proof}. 

Gouv\^ea's work, in particular, focused on a $p$-adic analogue of the vertical Sato--Tate phenomenon. He proposed the following conjecture on the distribution of $v_p(a_p(f))$ as $k \rightarrow \infty$.
\begin{conj}\label{conj:gouvea-intro}
Let $p$ be a prime number and assume $p \nmid M$. Define
\begin{equation*}
\mathbf x_T = \left\{\frac{p+1}{k} \cdot v_p(a_p(f)) ~\Big|~ \text{$f$ is a $\Gamma_0$-eigenform of weight  $k \leq T$} \right\}.
\end{equation*}
Then, $\mathbf x_T$ becomes equidistributed on $[0,1]$ for Lebesgue measure as $T\rightarrow \infty$.
\end{conj}
For accuracy's sake, we note that Conjecture \ref{conj:gouvea-intro} is only called a {\em question} by Gouv\^ea, as it was based largely on data gathered only for the smallest primes and weights (\cite[p.\ 3]{Gouvea}). Additionally, Gouv\^ea's paper is limited to eigenforms of level $\SL_2(\Z)$. However, significant cases of Conjecture \ref{conj:gouvea-intro} are known now (see Theorem \ref{thm:LTXZ-application} below) and so it seems apt to use the label ``conjecture''. We refer to Conjecture \ref{conj:gouvea-intro} simply as ``Gouv\^ea's distribution conjecture'' and take responsibility if its literal statement turns out to be false.

Let us review numerical evidence toward Conjecture \ref{conj:gouvea-intro}. Consider the prime $p=2$ and eigenforms of level $\SL_2(\Z)$. Figure \ref{fig:2_slopes_intro} is a scatter plot of $v_2(a_2)$ where $a_2=a_2(f)$ is sampled for eigenforms $f \in S_k(\SL_2(\Z))$ and $k \leq 512$. In the plot, the variable size of a point indicates the multiplicity of that point. 

\begin{figure}[htbp]
\centering
\includegraphics[scale=.75]{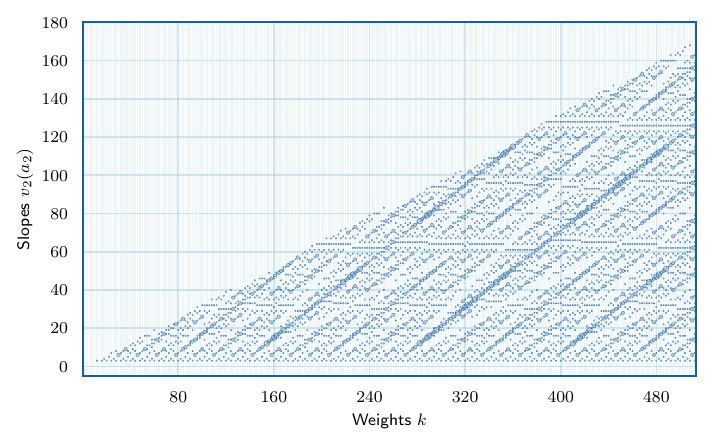}
\caption{$2$-adic slopes in level $\SL_2(\Z)$ for $k\leq 512$. (Total data:\ 5334 eigenforms.)}
\label{fig:2_slopes_intro}
\end{figure}

Figure \ref{fig:2_gouvea_dist} presents evidence for Gouv\^ea's distribution conjecture. Namely, the left-hand portion of that figure is a scatter plot for the Gouv\^ea-normalized data
\begin{equation*}
\frac{3}{k}\cdot v_2(a_2) \in [0,+\infty],
\end{equation*}
and the right-hand portion of the figure is the scatter plot's projection onto a histogram. The distribution conjecture is that, as $k \rightarrow \infty$, the histogram becomes a solid blue rectangle supported on $[0,1]$. We note that there is no analogue of the Ramanujan--Petersson bound in the $p$-adic context. The support of the statistical distribution on the interval $[0,1]$ has not been conceptually explained.

\begin{figure}[htbp]
\centering
\includegraphics[scale=.75]{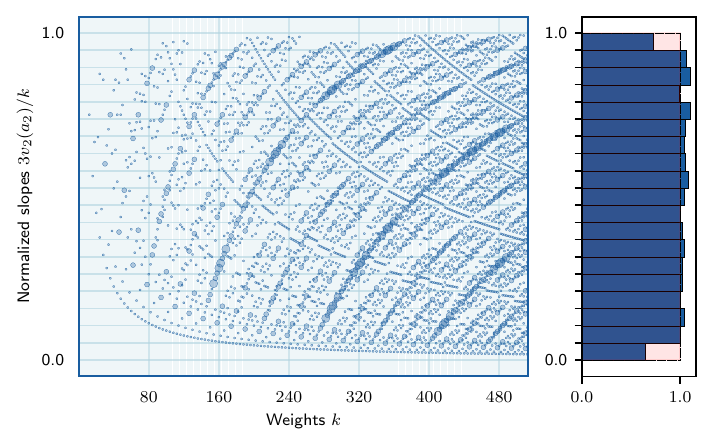}
\caption{Left:\ Normalized $2$-adic slopes in level $\SL_2(\Z)$ for $k\leq 512$.\\ Right:\ Histogram of normalized slopes. (Total data:\ 5334 eigenforms. Num.\ bins: 19.)}
\label{fig:2_gouvea_dist}
\end{figure}

Twenty years after Gouv\^ea's article, there is a positive result toward a refined version of Conjecture \ref{conj:gouvea-intro}, which we explain now. To each eigenform $f$, one may attach a semi-simple mod $p$ Galois representation
\begin{equation*}
\rhobar_f : \mathrm{Gal}(\Qbar/\Q) \rightarrow \GL_2(\Fpbar).
\end{equation*}
For any such $\rhobar$, we consider
\begin{equation*}
\mathbf x_T(\rhobar) = \left\{\frac{p+1}{k} \cdot v_p(a_p(f)) ~\Big|~ \parbox{4.75cm}{\centering $f$ is an $\Gamma_0$-eigenform of weight  $k \leq T$ such that $\rhobar_f \simeq \rhobar$} \right\}.
\end{equation*}
For a fixed $M$ there are only a finite number of $\rhobar$ arising from $\Gamma_0$-eigenforms. Therefore, one might hope that each $\mathbf x_T(\rhobar)$ is equidistributed on $[0,1]$, and then the original distribution conjecture is a finite sum of $\rhobar$-versions. The new result is a proof of the $\rhobar$-version for certain $\rhobar$. For notation, we denote by $\omega$ the mod $p$ cyclotomic character, and let $I_{\Qp}$ be the inertia subgroup of $\mathrm{Gal}(\Qpbar/\Qp)$. The following is one of six applications of a recent breakthrough on slopes by Liu, Truong, Xiao, and Zhao \cite{LTXZ-Ghost,LTXZ-Proof}.

\begin{thm}[{\cite[Theorem 1.21]{LTXZ-Proof}}]\label{thm:LTXZ-application}
Assume that $p \geq 11$. Let $\rhobar$ be irreducible and modular of level $\Gamma_0$, and
\begin{equation*}
\rhobar|_{I_{\Qp}} \simeq \begin{pmatrix} \omega^{a+1+b} & \ast \\ 0 & \omega^{b} \end{pmatrix},
\end{equation*}
with $2 \leq a \leq p-5$ and $b \in \Z$. Then, $\mathbf x_T(\rhobar)$ becomes equidistributed on $[0,1]$ for Lebesgue measure as $T \rightarrow \infty$.
\end{thm}

The most serious assumption in Theorem \ref{thm:LTXZ-application} is that the action of inertia is via powers of $\omega$ along the diagonal, as opposed to level two characters of inertia. Equivalently, that $\Gal(\Qpbar/\Qp)$ acts reducibly via $\rhobar$, as opposed to irreducibly. Modular Galois representations of level $\Gamma_0$ that satisfy this assumption are called {\em regular}. 

The proof of Theorem \ref{thm:LTXZ-application} is based on a proof of a more remarkable theorem on regular mod $p$ Galois representations.  Namely, we developed, between 2015 and 2019, a combinatorial model for slopes of $a_p$ that is called the {\em ghost series}. The {\em ghost conjecture} asserts that the ghost series predicts the slopes of weight $k$ eigenforms, via its Newton polygons, whenever $\rhobar$ is regular \cite{BP-Ghost,BP-Ghost2}. The main theorem of Liu--Truong--Xiao--Zhao  is that the ghost conjecture holds under the hypotheses listed in Theorem \ref{thm:LTXZ-application}. See \cite[Theorem 1.3]{LTXZ-Proof}. In this way, analyses of slopes are reduced to analyses for the combinatorially-defined ghost series, which are easier. In fact, the authors themselves proved Gouv\^ea's distribution conjecture is a corollary of their ghost conjecture. See \cite[Corollary 3.3]{BP-Ghost2}. Nonetheless, all but the smallest fraction of credit for Theorem \ref{thm:LTXZ-application} should be directed at Liu, Truong, Xiao, and Zhao, for their amazing work proving the ghost conjecture.

Alongside this breakthrough, interesting open questions remain related to Gouv\^ea's distribution conjecture. Of course, there is a strong desire to remove the ``generic'' assumption placed on the value of $a$ in Theorem \ref{thm:LTXZ-application}. A more fundamental problem is to establish Conjecture \ref{conj:gouvea-intro} in the case of irregular $\rhobar$, although such $\rhobar$ suffer from having no ghost series model upon which to base the analysis of slopes. 

One might also seek a {\em conceptual} understanding of the normalizing factor $\frac{p+1}{k}$, since we have such an understanding of the factor $p^{\frac{1-k}{2}}$ in the Archimedean cases. Normalizing by $k$ seems reasonable enough because there is a linear-in-$k$ number of eigenforms of weight $k$. Dividing by $k$ is also natural if one works at level $\Gamma_0(pM)$, where the largest $a_p$-slope is no more than $k-1$. (Gouv\^ea technically normalizes by $\frac{p+1}{k-1}$, even.) So, it is the factor of $p+1$ that remains obscure. Why is $v_p(a_p(f)) \leq \frac{k}{p+1}$ {\em almost} all the time? We emphasize ``almost'' here, as the proposed statistical bound does not always hold. After all, it can happen that $a_p(f) = 0$. (Calegari and Sadrari {\em did} recently prove that for fixed $M$ there are a finite number of eigenforms with vanishing $a_p$, excluding the forms with complex multiplication. See \cite[Theorem 1.1]{CalegariSardari}.)

In the rest of the paper, we reinforce Gouv\^ea's distribution conjecture by presenting a second instance of it.  Specifically, we turn to a set of $p$-adic statistics of eigenforms that are defined and calculated in a completely {\em different} way than $a_p$-slopes but which nonetheless seem to enjoy a remarkably similar distribution property. These are the $\L$-invariants of the article's title.

\subsection{The $\L$-invariants}\label{subsec:Linv-intro}

The discussions above dismissed the primes dividing the level $M$. One reason is that if $p \mid M$, then the possible values of $a_p(f)$ are completely understood.
\begin{enumerate}[label=(\roman*)]
\item If $p$ divides $M$ exactly once and $f$ has weight $k$, then $a_p(f) = \pm p^{\frac{k}{2}-1}$. Even more, the number of occurrences of $a_p(f) > 0$ versus $a_p(f) < 0$ in any given weight $k$ is bounded. (For this bounded-ness, see \cite[Corollary 1.4]{Martin-RootMurmurations} (or \cite[Section 3]{M} if $M$ is squarefree) or the Theorem \ref{thm:medvedetal} of Anni, Ghitza, and Medvedovsky.)
\item If $p^2 \mid M$, then $a_p(f) = 0$.
\end{enumerate}

The rest of our work focuses on the first case, which is of newforms of level $M = Np$ where $N$ is co-prime to $p$. The integer $N$ is called the tame level. Thus a newform $f$ of level $\Gamma_0$  has a weight $k$ and a discrete sign $\pm$. There is also a third special invariant of such a newform, called its $\L$-invariant $\L_f$. As far as it is known, it is impossible to calculate $\L_f$ directly from $q$-expansions or the local factors of the underlying automorphic representation. One can extract it, however, from related structures in the $p$-adic theory of automorphic forms:\ Galois representations, $p$-adic $L$-functions, and $p$-adic families of eigenforms. Understanding $\L$-invariants motivated some of the early advances within the $p$-adic Langlands program \cite{B}.

The most classical case of an $\L$-invariant arises from a weight two eigenform $f$ that corresponds to an rational elliptic curve $E$ with split multiplicative reduction at $p$.  In this case, the curve $E$ is a Tate curve over $\Q_p$, say with Tate parameter $q_E$. Choose the branch of the $p$-adic logarithm such that $\log_p(p) = 0$. Then, the $\L$-invariant of $E$ (or, equivalently, of $f$) was defined by Mazur, Tate, and Teitelbaum as
\begin{equation*}
\L_f = \L_E = \frac{\log_p(q_E)}{v_p(q_E)} \in \Q_p.
\end{equation*}
These $\L$-invariants are quickly calculated. For instance, if $j = j_E$ is the $j$-invariant of $E$, then $j^{-1} \in p\Z_p$ and $q_E$ can be expressed as an infinite series (see \cite[II.1]{MTT})
\begin{equation*}
q_E = \frac{1}{j} + \frac{744}{j^2} + \frac{750420}{j^3} + \dotsb \in p\Z_p.
\end{equation*}
From here, $\L_E$ can be found to any desired accuracy.

One definition of the $\L$-invariant that applies to a general $\Gamma_0$-newform is the Fontaine--Mazur $\L$-invariant. It is defined using Fontaine's $p$-adic Hodge theory for local Galois representations \cite{Mazur-Monodromy}. Its formal definition is recalled in Section \ref{subsec:Linvs-local-Galois}. Here, we only note its nature in parallel with $a_p$. For an eigenform $f$, let
$$
\rho_f : \Gal(\Qbar/\Q) \rightarrow \GL_2(\Qpbar)
$$
be the semi-simple (irreducible in fact, since $f$ is cuspidal) $p$-adic Galois representation attached to $f$. Let $r_f := \rho_f|_{\Gal(\Qpbar/\Qp)}$ be its restriction to a decomposition group at $p$. Then:

\begin{enumerate}[label=(\roman*)]
\item If $f$ has level $\Gamma_0(N)$, then $r_f$ may be uniquely described in terms of $k$ and $a_p(f)$.
\item If $f$ has level $\Gamma_0(Np)$, the representation $r_f$ may be uniquely described by $k$, the sign of $a_p(f) = \pm p^{\frac{k}{2}-1}$, and $\L_f \in \Qpbar$.
\end{enumerate}
So, from the perspective of moduli of Galois representations, the $a_p$-parameter is to eigenforms of level prime-to-$p$ as the $\L$-invariant is to eigenforms of level $\Gamma_0$.

There is a second connection between $a_p$ and $\L$-invariants. Namely, a famous formula first established by Greenberg and Stevens \cite{GS} in the elliptic curve case states
\begin{equation}
\label{eqn:ap}\tag{GS}
\L_f = -2\frac{a_p'(k)}{a_p(f)},
\end{equation}
where $a_p'(-)$ is the derivative of $a_p$ along the unique $p$-adic family of eigenforms passing through $f$ on Coleman and Mazur's eigencurve. (See \cite{Colmez-Linvariant} for a proof in terms of the Fontaine--Mazur $\L$-invariant.) For a fixed $\Gamma_0$-newform $f$, \eqref{eqn:ap} connects the $p$-adic magnitude of $\L_f$ to the variation of $a_p$ as a function on the eigencurve. Locally near $f$ on the eigencurve, all other eigenforms have level prime-to-$p$, and so the variation of slopes at level prime-to-$p$ influences $\L$-invariants at level $\Gamma_0$, and vice versa. The paper \cite{Bergdall-bounds} explores this theme further. See also \cite{CitraoGhateYasuda-Limit}.

To summarize, we have defined the $\L$-invariant of certain elliptic curves and indicated two general definitions (one could take \eqref{eqn:ap} as a definition). Further definitions, in historical order, are due to Teitelbaum \cite{T}, Coleman \cite{C}, Darmon \cite{D} and Orton \cite{O}, and Breuil \cite{B}. Some of these apply to any $\Gamma_0$-newform and some apply only in select situations. Starting in Section \ref{subsec:intro-ez}, we focus on the connection between $\L$-invariants and exceptional zeros of $p$-adic $L$-functions. That connection lies at the heart of the numerical computations reported on in this article. Before that, we summarize previously published computations.

\subsection{Prior computations}\label{subsec:history}

Despite the apparent interest in $\L$-invariants, with many researchers proposing different definitions and perspectives, there has been a definite deficit on explicitly {\it computing} them. Exactly two examples not coming from elliptic curves were computed in the seminal work of Mazur, Tate, and Teitelbaum. See \cite[pg.\ 46]{MTT}.\footnote{The second example appears to be off by a factor of $p=5$! (Exclamation, not factorial.)} In the decades following that study, only a few more examples were computed, in various works of Coleman and Teitelbaum, of Coleman, Stevens, and Teitelbaum, and of Lauder \cite{CT,CST,Lauder}.

Only more recently have robust algorithms been put forward. Gr\"af was the first (to our knowledge) to systematically publish data on $\L$-invariants \cite{G}.  He directly calculated using Teitelbaum's definition, which limits the data to eigenforms that arise via transfer from quaternionic modular forms. One computational advance applied in Gr\"af's approach is Franc and Masdeu's work on fundamental domains of the tree for $\GL_2(\Qp)$ (\cite{FM}). Anni, B\"ockle, Gr\"af, and Troya published further tables a short time later \cite{ABGT}. Their approach is a generalization of a method exposed by Coleman, Stevens, and Teitelbaum \cite{CST}. It involves  indirectly calculating $a_p'(k)/a_p(f)$ as in \eqref{eqn:ap}, by using the $p$-adic  variation of the $p$-th Hecke operator acting on overconvergent $p$-adic modular forms. The practicality of this approach is partly based on a computational advance by Lauder and by Vonk, who developed algorithms to calculate $p$-adic Hecke actions in time logarithmic in the weight, which is crucial for estimating $p$-adic derivatives \cite{Lauder,Vonk}.

The works of Gr\"af and Anni--B\"ockle--Gr\"af--Troya have limitations. For instance, they ``only'' calculate the list of $\L$-invariants arising in a fixed weight $k$ and with a fixed sign $a_p(f)=\pm p^{\frac{k}{2}-1}$. They cannot say which $\L$-invariant corresponds to which form. Their works also have distinct, beneficial, features. Gr\"af's work computes a little more than $\L$-invariants, since it also calculates a natural linear operator for which the $\L$-invariants are eigenvalues. (This is a special feature of Teitelbaum's definition.) The method via the Greenberg--Stevens formula notably presents the chance to perform calculations at level prime-to-$p$ as well. The ratio $-2a_p'(k)/a_p(f)$ is still called an $\L$-invariant then, however it is a purely global (even in Galois-theoretic terms) invariant. See \cite{Hida-GreenbergLInvariantAdjoint,Mok-LInvariant}, for instance. Such calculations could be relevant to questions on constant slope families raised in \cite{Bergdall-bounds}. To our knowledge, no one has systematically pursued such calculations.

\subsection{The exceptional zero method}\label{subsec:intro-ez}
This article focuses on practically calculating $\L$-invariants using $p$-adic $L$-functions. The departure point is the observation of Mazur, Tate, and Teitelbaum that the $p$-adic $L$-function of a eigenform $f$ of level $\Gamma_0$ has an {\em exceptional zero} at $s=k/2$ whenever the sign of $a_p(f)$ is positive. In that case, Mazur, Tate, and Teitelbaum {\em predicted} the existence of the $\L$-invariant as an invariant that would satisfy the relationship
\begin{equation}\label{eqn:ez}\tag{EZ}
L_p'(f,k/2) = \L_f \cdot L_\infty(f,k/2),
\end{equation}
and such that $\L_f$ would not change if $f$ is twisted by a Dirichlet character that is trivial at $p$. Here $L_p(f,s)$ is the $p$-adic $L$-function of $f$ and $L_\infty(f,k/2)$ is the algebraic part of the central $L$-value.  

Mazur, Tate, and Teitelbaum proposed a definition for $\L_f$ only in the case of weight two eigenforms, and Greenberg and Stevens established \eqref{eqn:ez} for such forms \cite{GS,GS2}. Later, Kato, Kurihara, and Tsuji proved \eqref{eqn:ez} based on the Fontaine--Mazur definition of $\L_f$. Their never-published proof is included in Colmez's Bourbaki survey \cite{Colmez-Bourbaki}. See Theorem \ref{thm:ez-theorem} below, or Colmez's survey for further historical discussion of \eqref{eqn:ez}.

The exceptional zero method to compute $\L_f$ means computing $L_p'(f,k/2)$ and $L_\infty(f,k/2)$, and then dividing one by the other.  The algebraic part of the central value can be computed using modular symbols. The $p$-adic $L$-function $L_p(f,s)$ could be calculated using {\em overconvergent} modular symbols as in \cite{PS1}. One may wonder why this section doesn't simply end with that observation.  In fact, there are some subtleties that one encounters in doing these computations.

First, algebraic central values can vanish. After all, whether or not $L_\infty(f,k/2) = 0$ is part of intensely studied questions such as Birch and Swinnerton-Dyer's conjecture on ranks of elliptic curves. When $L_\infty(f,k/2) = 0$, it is clearly impossible to calculate $\L_f$ as
$$
\L_f = \frac{L_p'(f,k/2)}{L_\infty(f,k/2)}.
$$
One may overcome this by replacing $f$ by a twist $f_\chi$ by a quadratic Dirichlet character $\chi$ trivial at $p$. The twist doesn't affect the $\L$-invariant, and $f_\chi$ will have a non-vanishing central $L$-value for some $\chi$. Incidentally,  twisting also allows us to use \eqref{eqn:ez} when $a_p$ is negative:\ if $a_p(f) < 0$, we could perform a twist by a character where $\chi(p) = -1$, instead. See Section \ref{subsec:centralvalues-quadratictwists} for a precise discussion.

Second, there are two issues with the approach ``using overconvergent modular symbols as in \cite{PS1}'' to compute $L_p'(f,k/2)$. First, the only completely developed computer programs are limited to eigenforms with coefficients defined over $\Qp$ and so a custom computation is needed. Second, before embarking on such a computation, one notes however that overconvergent modular symbols are overkill for the problem. Indeed, such a computation would produce the Taylor expansion of $L_p(f,s)$ around $s = k/2$. However, we do not need the entire $p$-adic $L$-function. We only need a single value of the first derivative!

To this end, in Section \ref{sec:estimate-derivative} we explain how to avoid using overconvergent modular symbols and instead rely on classical modular symbols. We first recall how to represent $L_p'(f,k/2)$ as the integral of a single power series on $\Zp^\times$. This integral can be approximated by Riemann sums from the classical modular symbol associated with $f$. The main theoretical result is Theorem \ref{thm:error_bound}, which is an estimate for the error term in such an approximation. Roughly, Riemann sums on balls on the form $a+p^n\Zp$ produce an error term bounded by $p^{e-nk/2}$ where $e = O(\log k)$, with constants independent of $n$. So, by taking $n \rightarrow \infty$, we could compute as many digits as we like, although the number of integrals calculated increases exponentially with $n$.

And now, perhaps, the reader is doubly confused. Overconvergent modular symbols provide a Taylor expansion of $L_p(f,s)$ in time linear in the number of $p$-adic digits, whereas the method of Riemann sums appears to be exponential. Here we encounter an unexpected, and helpful, phenomenon:\ in practice, we can almost always choose just $n = 1$. That is, estimates using Riemann sums over just the balls $a+p\Z_p$ are more than enough for the vast majority of the data. In fact, the $n=1$ approximation accurately gives roughly $k/2$ digits for $L_p'(f,k/2)$ and our data indicates $L_p'(f,k/2)$ has $p$-adic norm quite close to $1$, regardless of $k$. Thus $k/2$ digits of accuracy is more than enough to compute the $p$-adic valuation of the $\L$-invariant for all but the smallest of weights.

We implemented the above sketch in Magma \cite{magma}. See the github repository \cite{Linv-github}. The method relies on constructing modular eigensymbols representing the Galois orbits of newforms. This is a time consuming process, as the weight grows. However, the method has a major benefit over those described in Section \ref{subsec:history}. Namely, by gathering data eigenform-by-eigenform, we are able to line $\L$-invariants up alongside global arithmetic information, such as the congruence classes of Galois representations modulo $p$. Therefore, we can more clearly analyze our data through the lens of phenomena in the $p$-adic Langlands program.

\subsection{Data analysis}

We now summarize the findings indicated by our data. Figure \ref{fig:2_2_raw_huge} provides a sample plot. The horizontal axis represents even integer weights up to just beyond $800$. The vertical axis represents the $2$-adic valuations $v_2(\L_f)$ of level $\Gamma_0(2)$-newforms. This data set includes over thirteen thousand eigenforms. As in Figure \ref{fig:2_slopes_intro}, the size of a scatter point represents a multiplicity in the data.

\begin{figure}[htbp]
\centering
\includegraphics[scale=.75]{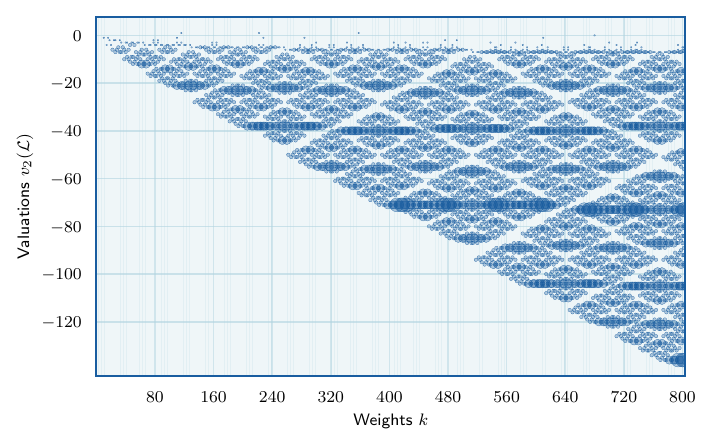}
\caption{$2$-adic valuations of $\L$-invariants in level $\Gamma_0(2)$ with respect to weights $10 \leq k \leq 804$. (Total data:\ 13465 eigenforms.)}
\label{fig:2_2_raw_huge}
\end{figure}

Many aspects of Figure \ref{fig:2_2_raw_huge} are striking:\ the clear triangular shape; the diagonal bands of white space; the thick bands of high-multiplicity data stretching horizontally across the plot; the $2\times 2$ diamond-shaped chambers; and so on. As you can see from the vertical scale in Figure \ref{fig:2_2_raw_huge}, our computations suggest that nearly all $\L$-invariants of newforms have $p$-adic denominators. We welcome this unexpected phenomenon, since we explained in the prior section that it makes our computations as fast as they could be. But why does it happen? The structures seen in the scatter plot demand explanations.

There are also important properties that cannot be seen in Figure \ref{fig:2_2_raw_huge}. First, for fixed $v$, the multiplicities of $v=v_2(\L)$ are evenly distributed between eigenforms with $a_2(f) = +2^{\frac{k}{2}-1}$ versus eigenforms with $a_2(f) = -2^{\frac{k}{2}-1}$. Second, the plot of Figure \ref{fig:2_2_raw_huge} occurs at level $\Gamma_0(2)$, a level in which the only mod 2 Galois representation is $1 \oplus 1$. If we work with different primes $p$ and different levels $N$, there will be multiple Galois representations appearing. What we see is that the basic qualities of the plots persist both when examining all eigenforms of level $\Gamma_0(Np)$ {\em or} filtering the data according to a fixed Galois representations modulo $p$. It seems natural, then, to dive deeper into filtering the data by a fixed mod $p$ Galois representation.

To explain the perspective on Galois representations, consider a fixed global representation 
\begin{equation*}
\rhobar : \Gal(\Qbar/\Q) \rightarrow \GL_2(\Fpbar),
\end{equation*}
modular of level $\Gamma_0(N)$, and its local restriction  $\rbar = \rhobar|_{\mathrm{Gal}(\Qpbar/\Qp)}$. For simplicity, assume $\rbar$ has only scalar endomorphisms. (See Section \ref{subsec:rbar-loci} for further discussion.) Given $\rbar$, a weight $k \geq 2$, and a sign $\pm$, we consider the set $X_k^{\pm}(\rbar) \subseteq \P^1(\Qpbar)$ consisting of $\L$-invariants that arise from semi-stable, but non-crystalline, lifts of $\rbar$ with weight $k$ and sign $\pm$. The set $X_k^{\pm}(\rbar)$ is the $\Qpbar$-points of the rigid generic fiber of a certain deformation ring, embedded into $\P^1$ by means of the $\L$-invariant.

The global data we gathered represents a finite sampling of $X_k^{\pm}(\rbar)$, as we vary over weights and signs in a fixed level. Thus, the shape of our data is constrained by the sets $X_k^{\pm}(\rbar) \subseteq \P^1(\Qpbar)$. A local-global principle in modularity lifting theorems (Theorem \ref{thm:local-global}) allows us to turn this constraint around and perceive the data itself to be a faithful representation of $X_k^{\pm}(\rbar)$. We first learned of this idea from a survey article of Buzzard and Gee on the slope problem \cite{BG-Survey}.

One notable connection we make in Section \ref{sec:reanalysis} is that the ways in which the structure of our data reflect geometric structures on Emerton and Gee's moduli stack of Galois representations \cite{EG}. Rather than fix $\rhobar$ and study the corresponding $\L$-invariant data, we  treat $\rhobar$ (or better yet, $\rbar$) as a variable over the moduli stack and investigate the extent to which the sets $X_k^{\pm}(\rbar)$ arrange themselves as $\rbar$ varies of the stack.

\subsection{The distributional conjecture}

Finally, we propose a statistical law for $\L$-invariants. It is a close cousin to to Gouv\^ea's distribution conjecture and we propose that both conjectures should be integrated into the circle ideas surrounding the Sato--Tate distribution and its vertical analogue.

The diagonal white bands in Figure \ref{fig:2_2_raw_huge} indicate that for each weight $k$, some values of $v_p(\L)$ are systematically avoided. Despite this pattern, we find the data {\em does} distribute itself evenly after normalizing the plot's triangular shape, even after restricting signs and Galois representations modulo $p$. To set notation, let $S_k(\Gamma_0)^{\pm}$ be the space of cuspforms in $S_k(\Gamma_0)$ that are new at $p$ and and for which $a_p = \pm p^{\frac{k}{2}-1}$.

\begin{conj}\label{conj:intro-distribution}
Let $N$ be an integer co-prime to $p$ and $\Gamma_0 = \Gamma_0(Np)$. Let $\rhobar: \Gal(\Qbar/\Q) \rightarrow \GL_2(\Fpbar)$ be modular of level $\Gamma_0(N)$, and fix a sign $\pm$. Define

\begin{equation*}
\mathbf y_T^{\pm}(\rhobar) = \left\{\frac{2(p+1)}{k(p-1)} \cdot v_p(\L_f) ~\Big|~ \parbox{4.75cm}{\centering $f \in S_k(\Gamma_0)^{\pm}$ is an eigenform,  $k \leq T$, and $\rhobar_f \simeq \rhobar$} \right\}.
\end{equation*}
Then, $\mathbf y_T^{\pm}(\rhobar)$ becomes equidistributed for Lebesgue measure on $[-1,0]$ as $T \rightarrow \infty$.
\end{conj}

When $p = 2$, the normalization in Conjecture \ref{conj:intro-distribution} is by $\frac{6}{k}$. (The slope in Figure \ref{fig:2_2_raw_huge} is about $-\frac{1}{6}$.)  Figure \ref{fig:2_2_dist_huge} presents a scatter plot and histogram to support Conjecture \ref{conj:intro-distribution}. In level $\SL_2(\Z)$, the only mod $p=2$ Galois representation is $\rhobar = 1 \oplus 1$. We do not filter by the sign $\pm$ in the plot itself. See Section \ref{subsec:plot-dist} for more plots.

\begin{figure}[htbp]
\centering
\includegraphics[scale=.75]{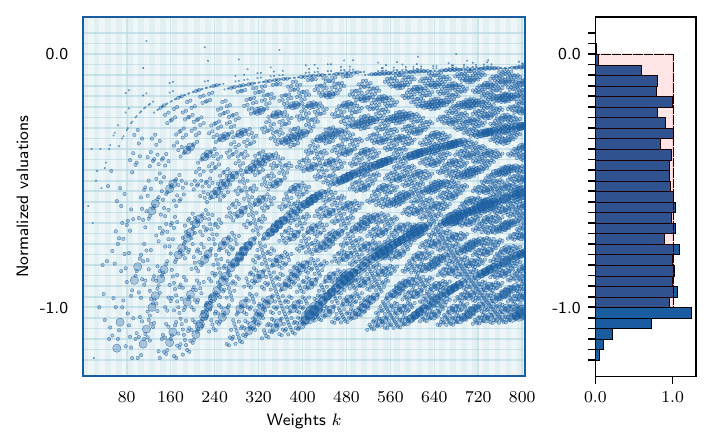}
\caption{Left:\ Normalized $\frac{6}{k}v_2(\L_f)$ in level $\Gamma_0(2)$ for $10\leq k\leq 804$.\\Right:\ Histogram of normalized slopes. (Total data:\ 13465 eigenforms. Num.\ bins = 31.)}
\label{fig:2_2_dist_huge}
\end{figure}

Figure \ref{fig:2_2_dist_huge} may make the reader nervous, as there seems to be significant bias away from $v_2(\L)=0$ and towards $v_2(\L)=-1$. (In Gouv\^ea's distribution, the bias seem to present itself symmetrically at both ends of the distribution.) We found the same bias occurring in all contexts, except the magnitude of the bias seems to decrease as $p$ increases. In an initial defense of Conjecture \ref{conj:intro-distribution}, it does seem that the bias in Figure \ref{fig:2_2_dist_huge} is disappearing as $k\rightarrow +\infty$. A better defense is that we can also support Conjecture \ref{conj:intro-distribution} with a heuristic that {\em predicts} the equidistribution phenomenon. In fact, as the scientific method demands, we used our heuristic to derive the normalizing factor $\frac{2(p+1)}{p-1}$ in Conjecture \ref{conj:intro-distribution} long before major data gathering efforts were carried out. The heuristic is explained in Section \ref{subsec:heuristic}. 

Due to good timing, we can even provide near to the {\em best} defense. While finishing the drafts of this article, we learned that Jiawei An has proven Conjecture \ref{conj:intro-distribution} (excluding the refinement on the sign of $a_p$) whenever the authors' ghost conjecture holds (for instance, by \cite{LTXZ-Proof}, under the hypotheses of Theorem \ref{thm:LTXZ-application}).  See \cite[Theorem 1.5]{An}. An's method is to assume the ghost conjecture holds and (a) relate the valuation of the $\L$-invariant of a newform to the size of the maximal constant slope family passing through that form as partially predicted in \cite{Bergdall-bounds}, and (b) prove that the sizes of these families are distributed as predicted by Conjecture \ref{conj:intro-distribution}.

\subsection{Final questions}

The data we gathered continues to astound us. At the end of the article (see Section \ref{sec:other}) we raise three points that we hope will lead to future investigations. In order:
\begin{enumerate}
\item What is going on with the horizontal masses of data appearing in Figure \ref{fig:2_2_raw_huge}? These horizontal bands of high-multiplicity values for $v_p(\L_f)$ are unmistakable in that data set, along with all the others we compiled.
\item The exceptional zero formula relates $\L$-invariants to the ratio of algebraic central values and $p$-adic central derivatives. What does the distribution of $\L$-invariants in Conjecture \ref{conj:intro-distribution} say about the ($p$-adic!) distribution of those numbers?
\item From the Galois-theoretic perspective, one of the interesting phenomena revealed by our data set is that $v_p(\L_f)$ is typically a whole number, provided $f$ is a so-called regular eigenform. That is not {\em completely} true, as we explain at the end of the article.
\end{enumerate}

One final question we leave unanswered is how to understand Conjecture \ref{conj:intro-distribution} and Conjecture \ref{conj:gouvea-intro} as common instances of the same principle. As explained in Section \ref{subsec:Linv-intro} above, $p$-adic questions about $a_p$ or $\L$-invariants can be simultaneously realized within the context of Galois representations. An interesting project would be to come up with an {\em a priori} reason for the normalizing constant $\frac{p+1}{k}$ or $\frac{2(p+1)}{k(p-1)}$ in either case. Indeed, in the Sato--Tate conjecture, its vertical analogue, or their generalizations, there is an entire conceptual apparatus producing measures for equidistribution.  It would be fascinating to find a structure on deformation spaces of Galois representations that would explain the statistical normalizations.
We note that in Section \ref{subsec:heuristic} we {\em do} link the normalizing constants to each other by explaining how to derive the constant in Conjecture \ref{conj:intro-distribution} from that of Conjecture \ref{conj:gouvea-intro}. We just do not give an independent explanation for either normalization.

\subsection{Article organization}

The body of the article is organized into eight sections. In Sections \ref{sec:background}-\ref{sec:summary} we explain background on modular symbols and $p$-adic $L$-functions and the practical method for calculating $\L$-invariants via the exceptional zero method. In Section \ref{sec:data}, we present the raw data on $\L$-invariants and discuss many of the features mentioned above. Section \ref{sec:theory} recalls the definition of the $\L$-invariant in terms of Galois representations and more generally introduces Galois-theoretic perspectives. Section \ref{sec:reanalysis} then reanalyzes the data from that perspective. In Section \ref{sec:conjecture} we introduce and study the distributional Conjecture \ref{conj:intro-distribution}. Finally, in Section \ref{sec:other} we raise a few questions for future research.

\subsection{Acknowledgments}
We thank Kevin Buzzard, Toby Gee, Peter Gr\"af, and Jiawei An for discussions and comments related to this research. We especially thank An for sending us a draft  of his preprint \cite{An}.

The research reported on here was carried out, partially, while the authors visited the Max-Planck Institute for Mathematics in Bonn, Germany. The staff and members of the institute are warmly thanked for their support and hospitality. Bergdall's research has been partially supported by Simons Foundation Travel Support Grant for Mathematicians MPS-TSM-00713782 and National Science Foundation grant DMS-2302284. Pollack's research has been partially supported by National Science Foundation grants DMS-1702178 and DMS-2302285 and by Simons Foundation Travel Support Grant for Mathematicians MPS-TSM-00002405.

\section{Background}\label{sec:background}

In this section we briefly recall modular symbols, $p$-adic $L$-functions, central values of $L$-functions, and the exceptional zero conjecture.

\subsection{Notations}

We fix an integer $N \geq 1$ and a prime number $p \nmid N$. We write $\Gamma_0 = \Gamma_0(Np)$ and $S_k(\Gamma_0)$ for the space of cuspforms of level $\Gamma_0$. A newform $f \in S_k(\Gamma_0)$ is thus a newform of level $Np$.  

We fix an isomorphism $\iota: \C \simeq \Qpbar$, thus allowing us to view an eigenform as defined over $\Qpbar$. On $\Qpbar$, we also fix a $p$-adic valuation $v_p$ such that $v_p(p) = 1$. We always measure $p$-adic numbers with $v_p(-)$. If we apply $v_p(-)$ to a complex number, we mean to implicitly apply $\iota$ first.

\subsection{Modular symbols and $p$-adic $L$-functions}\label{subsec:modular-symbols}
Attached to a modular form $f \in S_k(\Gamma_0)$, we have modular integrals
$$
\phi_f(P,r) := 2 \pi i \int_{\infty}^r f(z) P(z) \dz
$$
where $r \in \Q$ and $P(z)$ is an polynomial over $\Z$ of degree at most $k-2$ (as in \cite{MTT}).  The modular integrals of an eigenform $f$ may be decomposed into plus and minus parts and renormalized to take algebraic values:
$$
\eta_f^\pm(P,a,m) := \frac{\phi_f(P(z),a/m)) \pm \phi_f(P(\pm z),\pm a/m)}{\Omega_f^\pm} \in \Qbar.
$$
Here $\Omega_f^\pm \in \C^\times$ are normalized in the following sense:\  we insist $v_p(\eta_f^{\pm}(P,a,m))$ is always non-negative, and for some choice of $P$, $a$, and $m$, this valuation vanishes.  We will refer to $\Omega_f^\pm$ as normalized periods.

Set $\eta_f = \eta_f^+ + \eta_f^-$. The modular symbol attached to $f$ is $\lambda_f$, defined by
$$
\lambda_f(P,a,m) := \eta_f( P(mz+a), -a, m).
$$
This modular symbol is readily computed by computer algebra systems such as Magma, Pari-GP and Sage.  (We implemented certain calculations in Magma \cite{magma}. See Section \ref{sec:summary}.) 

The $p$-adic $L$-function is built from the modular symbol.  Namely, if $f$ is new at $p$, set
\begin{equation}
\label{eqn:mu}
\mu_f(P,a,p^n) := a_p(f)^{-n} \cdot \lambda_f(P,a,p^n).
\end{equation}
It is a theorem that there exists a unique locally analytic distribution $\mu_f$ on $\Zpx$ with the property that 
\begin{equation}
\label{eqn:periods}
\int_{a+p^n\Zp} P(x) \d\mu_f(x) = \mu_f(P,a,p^n)
\end{equation}
for all $P(z)$ of degree less than or equal to $k-2$. The distribution $\mu_f$ {\em is} the $p$-adic $L$-function of $f$. 

We note that $p$-adic $L$-functions can also be defined for eigenforms that are old at $p$, but in that case one also needs to choose a root of the Hecke polynomial $x^2-a_p(f)x+p^{k-1}$.  When this root has valuation strictly less than $k-1$, then the construction of $\mu_f$ follows the same approach except that the definition of $\mu_f$ in \eqref{eqn:mu} is slightly different and depends upon this choice of root.  We will not discuss $p$-adic $L$-functions of forms that are old at $p$ any further.

Let us describe the $p$-adic $L$-function in the  ``$s$-variable". First, write each $x \in \Z_p$ uniquely as $x = \langle x \rangle \omega(x)$, where $\langle x \rangle \in 1 + p\Z_p$ and $\omega(x)$ is a $(p-1)$-st root of unity. Then, for a tame character $\psi :\Zpx \to \Fpx \to \Qpbar^\times$, we set
$$
L_p(f,\psi,s) := \int_{\Zpx} \psi(x) \langx^{s-1} \d\mu_f(x).
$$
This is the ``$\psi$-branch'' of the $p$-adic $L$-function and here $\psi$ has $p-1$ choices:\ $1$, $\omega$, $\omega^2$, $\dots$, $\omega^{p-2}$. Nonetheless, for $j$ an integer, we often abuse notation and simply write
$$
L_p(f,j) = \int_{\Zpx} x^{j-1} \d\mu_f(x)
$$
without mentioning a specific branch.  Note though that 
$$
\int_{\Zpx} x^{j-1} \d\mu_f(x) = \int_{\Zpx} \omega(x)^{j-1} \langx^{j-1} \d\mu_f(x) = L_p(f,\omega^{j-1},j)
$$
and thus $L_p(f,j)$ is the value at $j$ on the $\omega^{j-1}$-branch.  

We also drop the branch in derivatives of the $p$-adic $L$-function:\
\begin{equation}\label{eqn:derivative-derivation}
L_p'(f,j) := \frac{d}{ds}\bigg\vert_{s=j} \left( \int_{\Zpx} \omega(x)^{j-1} \langx^{s-1} d\mu_f(x)
\right).
\end{equation}
In Section \ref{sec:estimate-derivative}, we further compute \eqref{eqn:derivative-derivation} when $f$ has weight $k$ and $j=k/2$.

We close this subsection with a fact about the valuation of certain values of $\mu_f$, a fact which will be key in providing error estimates of our approximations of $p$-adic $L$-functions.

\begin{lemma}
\label{lemma:betterMTT}
Let $f \in S_k(\Gamma_0)$ be an eigenform that is new at $p$. For $m \geq 0$, we have
$$
v_p \left( \int_{a+p^n\Zp} (x-a)^m \dmuf \right) \geq n\left(m - \fk\right).
$$
\end{lemma}

\begin{proof}
This divisibility is almost \cite[pg.\ 13, III]{MTT}, but in that reference a lower bound of $n\left(m - \fk\right) -\fk$ is given.  However, in our situation of a form whose level is divisible by $p$ exactly once, we can achieve the slightly stronger bound of this lemma. Indeed, if $f$ had level prime-to-$p$ and $\alpha$ was a root of the Hecke polynomial of $f$ at $p$, then the very definition of $\mu_f$ yields that $\alpha^{n+1} \mu_f(P,a,p^n)$ is integral.  However, in our case, (\ref{eqn:mu}) implies that $\alpha^n \mu_f(P,a,p^n)$ is integral where $\alpha =  a_p(f) = \pm p^{\frac{k}{2}-1}$. Tracing through the proof of \cite[pg.\ 13, III]{MTT} using this stronger fact yields the claim of this lemma.
\end{proof}

\subsection{Central values and quadratic twists}\label{subsec:centralvalues-quadratictwists}
For $1 \leq j \leq k-1$, set
$$
L_\infty(f,j) = \eta_f(z^j,0,1).
$$
Thus $L_\infty(f,j) \in \Qbar$, which is viewed in $\Qpbar$ via $\iota$. It is called an algebraic special value, the most important of which is the {\em central} value at $j = k/2$. The following exceptional zero formula expresses the $\L$-invariant of an eigenform in terms of the central values of its $p$-adic $L$-function and complex $L$-series.

\begin{thm}\label{thm:ez-theorem}
Let $f \in S_k(\Gamma_0)$ be an eigenform that is new at $p$. Then, there exists $\L_f \in \Qpbar$ with the following two properties.
\begin{enumerate}
\item \label{part:twist} If $\chi$ is a quadratic Dirichlet character such that $\chi(p) = 1$, then $\L_{f} = \L_{f_\chi}$. (Here, $f_\chi$ is the quadratic twist of $f$ by $\chi$.)
\item If $a_p(f) = +p^{\frac{k}{2}-1}$, then $L_p(f,k/2) = 0$ and
\begin{equation}
\tag{EZ}\label{thm:EZ}
L_p'(f,k/2) = \L_f \cdot L_\infty(f,k/2).
\end{equation}
\end{enumerate}
\end{thm}

\begin{proof}
Colmez's Bourbaki survey on the $p$-adic Birch and Swinnerton-Dyer conjecture (\cite{Colmez-Bourbaki}) is one place to find a proof of this result. Specifically, in Section 4.6 of {\em op.\ cit.}, one finds the definition of the Fontaine--Mazur $\L$-invariant $\L_f$, which satisfies part (1) more or less by construction. (See Section \ref{subsec:Linvs-local-Galois}.) Then part (2) is proven in Th\'eor\`eme 4.16 of {\em op.\ cit.} The proof of (2) is due, in this context, to Kato, Kurihara, and Tsuji. (The careful reader will notice Colmez uses $f^{\ast}$ for the complex conjugate of $f$, which is just $f$ in our setting because $f$ has level $\Gamma_0$ and thus has real $q$-series coefficients.)
\end{proof}

\begin{remark}
The formulation of \eqref{part:twist} requires letting the tame level $N$ change.
\end{remark}

The obstructions to calculating  $\L$-invariants using \eqref{thm:EZ} are that the complex $L$-value  could vanish and $a_p(f)$ could be negative ({\it i.e.}\ equal to $-p^{\frac{k}{2}-1}$). However, by \eqref{part:twist}, we can always simultaneously twist ourselves out of these two situations.  Indeed, if $\chi$ is a quadratic Dirichlet character unramified at $p$, then $f_\chi$  is also an eigenform of level $\Gamma_0$, except possibly with a different tame level and $a_p(f_\chi) = \chi(p)a_p(f)$.    The following proposition guarantees the existence of a twist of $f$ whose $p$-th Fourier coefficient is positive and whose central $L$-value is non-zero. 

\begin{prop}\label{prop:twist-our-worries-away}
For each eigenform $f$ of level $\Gamma_0$ and weight $k\geq 2$, there exists a quadratic Dirichlet character $\chi$ unramified at $p$ such that $a_p(f_\chi) = +p^{\frac{k}{2}-1}$ and $L_\infty(f_\chi,k/2) \neq 0$.
\end{prop}
\begin{proof}
Let $w_f$ be the sign of the functional equation of $f$.
For a fundamental discriminant $D$, set 
\begin{equation*}
\chi_D(\cdot) = \left(\frac{D}{\cdot}\right).
\end{equation*}
 Twisting by $\chi_D$ changes the sign of the functional equation for $f$ by $\chi_D(-Np)$ (see \cite[(5.9)]{Bump}).  

We first seek a twist of $f$ having $+1$ as the sign of its functional equation and having $+p^{k/2-1}$ as its $U_p$-eigenvalue.  Said another way, we need a $D$ co-prime to $N$ such that
\begin{equation}
\label{eqn:sign}
\chi_D(-Np)=w_f \text{~~and~~} \chi_D(p) =\text{sign~of~} a_p(f).
\end{equation}
For any $N$, we can find infinitely many $D$ that satisfy (\ref{eqn:sign}), and for any such $D$, the twist $f_{\chi_D}$ has positive $a_p$. Then, among the infinitely many choices of $D$, there is one for which $L_\infty(f_{\chi_D},k/2) \neq 0$ by the main theorem of \cite{BFH}.
\end{proof}

In practice, to calculate $\L$-invariants, we first  calculate the modular symbols $\lambda_f$ for $\Gamma_0$-eigenforms $f$. This is a rather expensive step.  Luckily, if we ever need to apply Proposition \ref{prop:twist-our-worries-away}, we do not need to re-do an expensive calculation. Indeed, the modular symbol for $f_{\chi}$ can be readily computed from the symbol for $f$, by following formula (see \cite[(8.5)]{MTT}):\ if $\chi$ has conductor $D$, then
\begin{equation}\label{eqn:birch-stevens-twist}
\lambda_{f_\chi}(P(Dz),a,m) = \frac{1}{\tau(\chi)}  \sum_{b \bmod D} \chi(b) \cdot \lambda_f(P,Da-mb,Dm).
\end{equation}
This incidentally explains how to search for a $\chi$ satisfying Proposition \ref{prop:twist-our-worries-away}:\  calculate $\lambda_{f_{\chi}}(z^{k/2},0,1)$ for $D$ satisfying (\ref{eqn:sign}) until it does not vanish.

\section{Estimating the $p$-adic derivative  $L_p'(f,k/2)$}\label{sec:estimate-derivative}

In this section, we construct a sequence $\{L_n\}$ of explicit $p$-adic approximations for $L_p'(f,k/2)$. See also \cite[Proposition 3.1]{SW}, which treats the ordinary ($k=2$) case.

\begin{thm}
\label{thm:error_bound}
There exists an explicit sequence $\{L_n\}$  of linear combinations of period integrals of $f$ such that
$$
v_p(L'_p(f,k/2)-L_{n}) \geq \frac{nk}{2} - \left\lfloor \frac{\log (k-1)}{\log(p)} \right\rfloor. 
$$
\end{thm}

Note that taking $n=1$ in Theorem \ref{thm:error_bound} yields a bound around $\frac{k}{2}$. Thus, for all but the smallest weights, $L_1$ is already a good approximation of $L_p'(f,k/2)$. 

The remainder of the section is devoted to constructing the $L_n$ and proving Theorem \ref{thm:error_bound}.  We begin with a lemma which expresses $L_p'(f,k/2)$ in terms of $\mu_f$.

\begin{lemma}
\label{lemma:Lderiv}
We have
\begin{equation*}
L_p'(f,k/2) = \int_{\Zpx} x^{\frac{k}{2}-1} \cdot \log_p \langx \dmuf.
\end{equation*}
\end{lemma}

\begin{proof}
By \eqref{eqn:derivative-derivation} we need to calculate the derivative of $L_p(f,\omega^{\frac{k-2}{2}},s)$ at $s = k/2$. We compute
\begin{align*}
L_p(f,\omega^{\fk},s) 
&= \int_{\Zpx} \omega(x)^{\frac{k}{2}-1} \langx^{s-1} \dmuf\\
&= \int_{\Zpx} x^{\frac{k}{2}-1} \langx^{s-k/2} \dmuf\\
&= \int_{\Zpx} x^{\frac{k}{2}-1} \exp((s-k/2)\log_p \langx) \dmuf\\
&= \int_{\Zpx} x^{\frac{k}{2}-1} \sum_{n=0}^\infty \frac{1}{n!} (\log_p \langx)^n (s-k/2)^n \dmuf\\
&= \sum_{n=0}^\infty \frac{1}{n!} \left( \int_{\Zpx} x^{\frac{k}{2}-1} (\log_p \langx)^n \dmuf \right) (s-k/2)^n.
\end{align*}
Taking the $n = 1$ term completes the proof.
\end{proof}

By Lemma \ref{lemma:Lderiv}, 
$$
L_p'(f,k/2)=\ds \int_{\Zpx} g(x) \dmuf
$$ 
with $g(x)$ a locally analytic function on $\Zpx$.  To understand how to approximate $L_p'(f,k/2)$, let's recall how one extends $\mu_f$ from a functional on the space of locally polynomial functions of degree at most $k-2$ (as in (\ref{eqn:periods})) to a distribution on all locally analytic functions.

To this end, assume that $n$ is large enough so that $g(x)$ is analytic on $a + p^n\Zp$ for all $a$.  Then decompose $\Zpx$ as a disjoint union of balls of the form $a_i + p^n\Zp$ with each $a_i \in \Z$ (depending on $n$).  Write $\TS_{a_i}(g)$ for the Taylor series expansion of $g$ around $x=a_i$ {\em truncated} to terms of degree at most $k-2$.  One defines
\begin{align*}
\mu_f(g) 
&:= \lim_{n \to \infty} \sum_{i} \int_{a_i+p^n\Zp} \TS_{a_i}(g) \dmuf\\
&= \lim_{n \to \infty} \sum_{i} \mu_f(\TS_{a_i}(g),a_i,p^n)
\end{align*}
which indeed converges and the limit is independent of the choice of $a_i$.

Returning to the start of this section, we simply define
\begin{equation}\label{eqn:Ln-definition}
L_n = \sum_{\substack{a=1 \\ p \nmid a}}^{p^n-1} \mu_f(\TS_{a}(x^{\frac{k}{2}-1} \cdot \log_p \langx),a,p^n)
\end{equation}
which by Lemma \ref{lemma:Lderiv} and the above discussion converges to $L_p'(f,k/2)$ as $n \to \infty$. Since $L_n$ is defined by values of $\mu_f$ on polynomials of degree at most $k-2$, we see from Section \ref{subsec:modular-symbols} that $L_n$ is a linear combination of period of integrals of $f$, and thus each $L_n$ can be computed from the modular symbol attached to $f$.

Before giving a proof of Theorem \ref{thm:error_bound}, we need a technical lemma $p$-adically bounding the coefficients of the Taylor expansion of $x^{\frac{k}{2}-1} \cdot \log_p \langx$.

\begin{lemma}
\label{lemma:taylor_coefs}
If $a \in \Zp^\times$, then
$$
\ds x^{\frac{k}{2}-1} \cdot \log_p \langx = \sum_{m \geq 0} c_{m,a} (x-a)^m,
$$
where $c_{m,a} \in \Qp$ and $\ds v_p(c_{m,a}) \geq -\left\lfloor \frac{\log m}{\log p} \right\rfloor$.
\end{lemma}

\begin{proof}
The $n$-th derivative of $\log_p\langle x\rangle$ is $(-1)^{n-1} (n-1)! x^{-n}$. So, on $a+p^n\Zp$, we have the following Taylor expansion:
\begin{align*}
\log_p\langx 
&= \log_p\langle a \rangle + \sum_{j\geq1} \frac{(-1)^{j-1}}{j} a^{-j} (x-a)^j
\end{align*}
and thus
\begin{align*}
x^{\frac{k}{2}-1} \log_p\langx 
&= ((x-a)+a)^{\frac{k}{2}-1} \cdot \left( \log_p\langle a \rangle + \sum_{j\geq1} \frac{(-1)^j}{j} a^{-j} (x-a)^j \right).
\end{align*}
Since $a^{-j} \in \Zp^\times$ for any $j$, the only denominators in the above expression are the denominators $1/j$ appearing in the expansion of the logarithm.  In particular, expanding the above expression shows that its $m$-th coefficient is a sum of terms whose denominators are integers with magnitude at most $m$.  This observation implies the lemma.
\end{proof}

\begin{proof}[Proof of Theorem \ref{thm:error_bound}]
For $a \in \Zp^\times$, write $x^{\frac{k}{2}-1} \log_p \langx = \sum_{m \geq 0} c_{m,a} (x-a)^m$.  Then we have
\begin{align*}
L'_p(f,k/2)-L_n
&=
\int_{\Zpx} x^{\frac{k}{2}-1} \log_p \langx \dmuf - \sum_{\substack{a=1 \\ p \nmid a}}^{p^n-1} \mu_f(\TS_{a}(x^{\frac{k}{2}-1} \cdot \log_p \langx),a,p^n)\\
&=
\sum_{\substack{a=1 \\ p \nmid a}}^{p^n-1}
\int_{a+p^n\Zp} \sum_{m=k-1}^\infty c_{m,a} (x-a)^m \dmuf \\ 
&=
\sum_{\substack{a=1 \\ p \nmid a}}^{p^n-1}
 \sum_{m=k-1}^\infty c_{m,a} \int_{a+p^n\Zp} (x-a)^m \dmuf.
\end{align*}
Thus, using Lemma \ref{lemma:betterMTT} to estimate the integrals $\int (x-a)^m \dmuf$ and Lemma \ref{lemma:taylor_coefs} to estimate the $c_{m,a}$, we see
\begin{align*}
v_p(L'_p(f,k/2)-L_{n}) 
&\geq \min_{m \geq k-1} \left\{ n\left(m - \fk\right) - \left\lfloor \frac{\log m}{\log p} \right\rfloor \right\}\\
&=   \frac{nk}{2} - \left\lfloor \frac{\log(k-1)}{\log p} \right\rfloor 
\end{align*}
as desired.
\end{proof}

\section{Summary of algorithm to calculate $\L$-invariants}\label{sec:summary}

In this section, we summarize practical considerations for calculating valuations of $\L$-invariants. A practical version of the algorithm was implemented in Magma \cite{magma}. The code is posted to the github repository \cite{Linv-github}.

As explained in Section \ref{sec:background}, we use modular symbols rather than modular forms. Fix a sign $\varepsilon$ for complex conjugation. Let 
\begin{equation*}
H_k \subseteq H^1(Y_0(Np),\Sym^{k-2}(\Q^2))^{\varepsilon}
\end{equation*}
be the subspace that is both new of level $\Gamma_0$ and cuspidal. The outcome of the algorithm will be the $p$-adic valuations $v_p(\L)$ of the $\L$-invariants for the newforms of level $\Gamma_0$, detected through their realization in $H_k$.

\begin{enumerate}[label=(\roman*)]
\item Perform a decomposition 
$$
H_k = \bigoplus_{i=1}^d A_i
$$ 
into modules simple for the Hecke algebra. We then fix $A = A_i$, and from this point forward the calculation depends just on $A$. \label{alg:first-step}
\item The eigenvalues of $U_p$ acting on $A$ must all be of the form $\pm p^{\frac{k}{2}-1}$ with a fixed sign. We calculate the $p$-th Hecke polynomial and use it to determine the sign of $a_p = a_p(A)$. Since $A$ is new of level $\Gamma_0$, the sign $w = w(A)$ of the functional equation can extracted from the sign of the eigenvalue for the Atkin--Lehner operator $w_{Np}$ on $A$. Having determined these two signs, we find a quadratic discriminant $D$ as in Proposition \ref{prop:twist-our-worries-away}. Write $\chi=\chi_D$ for the quadratic character of conductor $D$.
\item Let $K$ be the number field generated by $A$. We construct a non-zero Hecke eigenvector $\lambda \in A \otimes_{\Q} K$. This is tantamount to choosing a modular form $f$ from the Galois orbit of eigenforms corresponding $A$. So, write $\lambda = \lambda_f$. We then pre-compute two sets of data.\label{alg:second-step}
\begin{enumerate}[label=(\alph*)]
\item We pre-compute the values\label{alg:second-step-S}
$$
\mathcal S = \{\lambda_f(z^i,x,y) \mid 0 \leq i \leq k-2\},
$$
for a collection of rational numbers $x/y$ such that $\{\infty\}-\{x/y\}$ are $\Z[\Gamma_0]$-generators of all unimodular paths.
\item We pre-compute\label{alg:second-step-est}
$$
\lambda_f(z^i, Da-pb,Dp) \in K
$$
where $1 \leq a \leq p-1$ and $1 \leq b \leq D$ is co-prime to $D$. 
\end{enumerate}\label{step:3rd-step}
\item In calculating the values of $\lambda_f$, a computer algebra system implicitly uses {\em some} choice of periods. However, the bound in Theorem \ref{thm:error_bound} depends on  fixing normalized periods as in Section \ref{sec:background}. We impose this normalization as follows. For each prime $\mathfrak p$ dividing $p$ in $K$, we consider the natural $\mathfrak p$-adic valuation $v_{\mathfrak p}$ on $K$, normalized so that $v_{\mathfrak p}(p) = 1$. Then, use \ref{alg:second-step}\ref{alg:second-step-S} to define
$$
\nu_{\mathfrak p} =  \min_{s \in \mathcal S} v_{\mathfrak p}(s).
$$
If normalized  periods were used, we would have $\nu_{\mathfrak p} = 0$. In general, $\nu_{\mathfrak p}$ represents the shift in valuation required to replace the (unknown) periods with the normalized ones.
\item \label{alg:final-step} Based on \ref{alg:second-step}\ref{alg:second-step-est}, we can calculate any summation $\lambda_{f_\chi}(-,-,-)$ as in \eqref{eqn:birch-stevens-twist}. In particular, we determine
\begin{equation*}
L_\infty (f_\chi,k/2)= \lambda_{f_{\chi}}(z^{\frac{k}{2}-1},0,1).
\end{equation*}
along with the $p$-adic estimate $L_1 \in K$ of $L_p'(f_\chi,k/2)$ in \eqref{eqn:Ln-definition}. Now fix a prime $\mathfrak p$ dividing $p$ in $K$. If
\begin{equation}\label{eqn:bound-needed}
v_{\mathfrak p}(L_1) - \nu_{\mathfrak p} > \frac{k}{2} - \left\lfloor \frac{\log (k-1)}{\log(p)} \right\rfloor,
\end{equation}
then, by Theorem \ref{thm:error_bound}, we see that
$$
v_{\mathfrak p}(L_1)-v_{\mathfrak p}(L_\infty(f_\chi,k/2))
$$
is equal to $v_{\mathfrak p}(\L)$ for $[K_{\mathfrak p}:\Q_p]$-many eigenforms occurring in $A$. So, we store this value, with multiplicity $[K_{\mathfrak p}:\Qp]$. If \eqref{eqn:bound-needed} fails for some $\mathfrak p$, then return to step \ref{alg:second-step-est} to compute at a deeper level, eventually replacing $L_1$ by $L_2$, and so on.
\end{enumerate}

\begin{remark}
At the start of this section, we said we made a {\em practical} implementation of the code. This means two things. First, the code assumes that $N$ is not a square. When $N$ is a square, we cannot control the value of $\chi(-1)$ for $\chi$ arising from Proposition \ref{prop:twist-our-worries-away}, and, in particular, we do not have {\it a priori} knowledge about the sign of the space of modular symbols we need to compute. We did make a custom computation to calculate in level $Np = 9\cdot 5 = 45$, where the data became useful to write Section \ref{subsec:galois}, by simply computing both the plus and minus spaces of modular symbols. Second, we did not implement the ``...and so on" step in \ref{alg:final-step}. Indeed, the initial data we gathered convinced us that 100\% of the data would pass step \ref{alg:final-step} from the start, so we simply throw away the thin set of data where \ref{alg:final-step} fails.
\end{remark}

\section{Raw data and observations}\label{sec:data}

In this section, we present and discuss data collected on $\L$-invariants. Section \ref{subsec:data-raw} contains the raw data plots in tame level one, for odd primes up to $p=11$. Their striking triangular nature is the topic of Section \ref{subsec:thresholds}. The remaining two sections discuss filtering the data, first by the sign of $a_p$ in Section \ref{subsec:signs} and then by associated Galois representations modulo $p$ in Section \ref{subsec:data-galois}. Phenomena related to the sign of $a_p$ has previously been noted in \cite{ABGT}, and our larger data collection efforts reinforce that article's observations. The discussion on filtering by Galois representations is novel to our approach.

We provide plots of only some of the data we collected. The complete set of raw data, and many other visualizations can be found in the github repository \cite{Linv-github}. The vast majority of the calculations were performed on a computing server at the Max Planck Institut f\"ur Mathematik. They have our gratitude once again for their hospitality and use of their machines.

\subsection{Initial plots in tame level one}\label{subsec:data-raw}

The next four figures present the data we gathered on valuations of $\L$-invariants in level $\Gamma_0(p)$ for primes $p=3,5,7,11$. (The prime $p=2$ is included as Figure \ref{fig:2_2_raw_huge} in the article's introduction.) The scatter plots below and similar plots throughout Section \ref{sec:data} are constructed with the following parameters.
\begin{itemize}
\item The horizontal axis represents {\em weights} of eigenforms. 
\item The vertical axis represents the {\em $p$-adic valuations} of $\L$-invariants, often called their {\em slopes}.
\item The size of a scatter point represents the datum's multiplicity. For instance if $v_p(\L) = -8$ occurs twice in weight $k$ and $v_p(\L) = -10$ occurs four times, then the scatter point at $(k,-10)$ is a magnitude larger than the scatter point at $(k,-8)$.
\end{itemize}
For each figure, the caption indicates the level at which the data was collected, along the total number of data represented ({\it i.e.}\ the total number of eigenforms).

\begin{figure}[htbp]
\centering
\includegraphics[scale=.75]{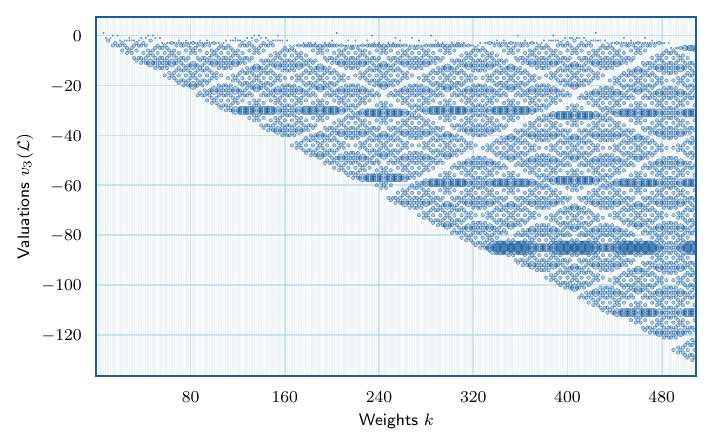}
\caption{$3$-adic slopes of $\L$-invariant slopes in level $\Gamma_0(3)$ with respect to weights $6 \leq k \leq 508$. (Total data:\ 10752 eigenforms.)}
\label{fig:3_3_raw}
\end{figure}

\begin{figure}[htbp]
\centering
\includegraphics[scale=.75]{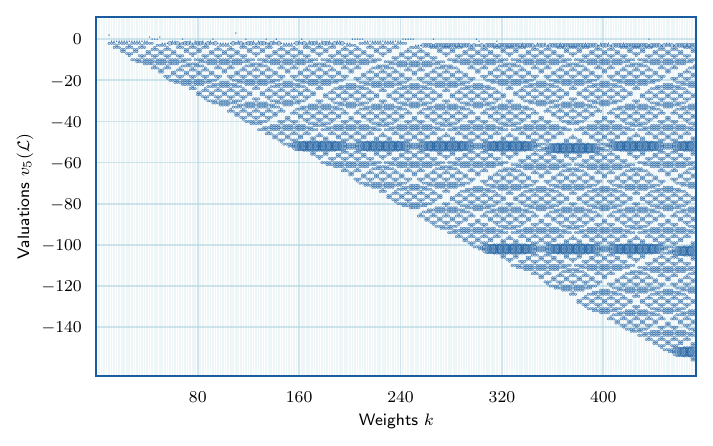}
\caption{$5$-adic slopes of $\L$-invariant slopes in level $\Gamma_0(5)$ with respect to weights $10 \leq k \leq 472$. (Total data:\ 18560 eigenforms.)}
\label{fig:5_5_raw}
\end{figure}

\begin{figure}[htbp]
\centering
\includegraphics[scale=.75]{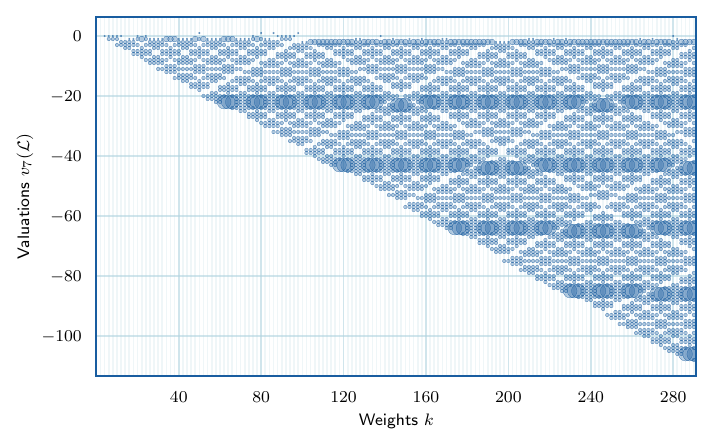}
\caption{$7$-adic slopes of $\L$-invariant slopes in level $\Gamma_0(7)$ with respect to weights $4 \leq k \leq 290$. (Total data:\ 10511 eigenforms.)}
\label{fig:7_7_raw}
\end{figure}

\begin{figure}[htbp]
\centering
\includegraphics[scale=.75]{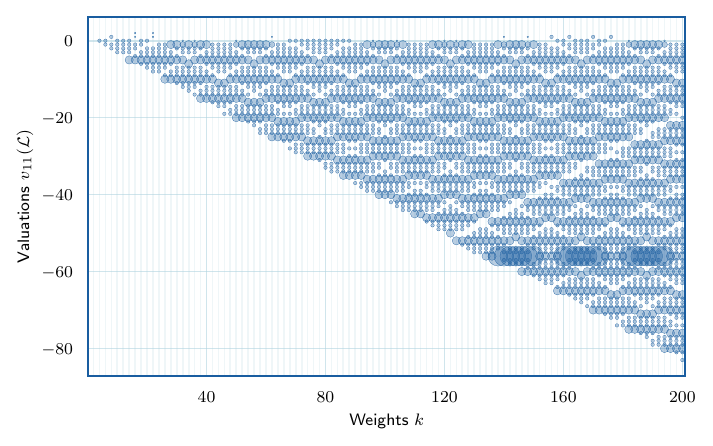}
\caption{$11$-adic slopes of $\L$-invariant slopes in level $\Gamma_0(11)$ with respect to weights $4 \leq k \leq 200$. (Total data:\ 8332 eigenforms.)}
\label{fig:11_11_raw}
\end{figure}

\subsection{Thresholds}\label{subsec:thresholds}

The triangular nature of the plots indicate two thresholds, the lower sloped side decreasing with respect to the weight and the upper horizontal threshold.

Let us first focus on the lower threshold. It seems there is a constant $C > 0$ such that the vast majority of $\L$-invariants satisfy $v_p(\L) \geq -Ck$. Using Figure \ref{fig:3_3_raw}, one can estimate that $C = \frac{1}{4}$ for level $\Gamma_0(3)$. Based on Figure \ref{fig:5_5_raw}, it seems that $C = \frac{1}{3}$ for level $\Gamma_0(5)$. Continuing the pattern, the data we have gathered supports, for all primes $p$ and all fixed levels $N$, the value
$$
C = \frac{p-1}{2(p+1)}
$$
as the relevant threshold.

Note that we are {\em not} proposing that  $v_p(\L) \geq -\frac{(p-1)k}{2(p+1)}$ for all eigenforms, as that is false. Rather, we are saying the number of times the bound is violated is statistically insignificant as $k\rightarrow +\infty$. A broader conjecture will be made later, in Conjecture \ref{conj}.

At the moment, there is no theoretical justification for the lower linear threshold. However, there are two kinds of evidence we can offer. First, there is {\em some} linear threshold for the data shown in these plots, though that threshold may not generalize to other primes or tame levels. See Theorem \ref{thm:BLL-intext} below, a theorem that will make more sense after discussing Galois representations in \ref{sec:theory}. The second evidence is that the precise form of $C$ predicted based on a heuristic (see Section \ref{subsec:heuristic}) and then verified from the scatter plots.

What about the upper threshold in our plots? Why should $v_p(\L)$ be negative almost all the time? Again, there is no theoretical evidence this happens, although it follows from Conjecture \ref{conj} once more. We can also say that this phenomenon (and likely the other threshold) is a by-product of studying $\L$-invariants for eigenforms in a fixed tame level with unbounded weights. For instance, among roughly 400 isogeny classes of rational elliptic curves that have split multiplicative reduction at $p=3$, and with conductor less than $10^3$, there are only 2 with $\L$-invariant having {\it negative} valuation.

\subsection{The sign of $a_p$}
\label{subsec:signs}

The $\L$-invariant data can be filtered into two parts based on the sign of $a_p = \pm p^{\frac{k}{2}-1}$. The overlap between the two resulting data sets is impossible not to notice. 

For instance, in level $\Gamma_0(3)$ and weight $k = 36$, the $3$-adic $\L$-invariants have valuations
$$
v_3(\L) = -9,-9,-4,-4,-2.
$$
These split as evenly as they could among the two signs:\ the $-9$ occurs once for each sign and the $-4$ occurs once for each sign. The $-2$ happens to occur for an eigenform with $a_3<0$. For a second example, in weight $k=28$ and level $\Gamma_0(3)$, there are an even number of eigenforms and their $\L$-invariants have slopes
$$
v_3(\L) = -6,-6,-2,-1.
$$
In this case, the $\{-6,-1\}$ appears with a positive $a_p$-sign and $\{-6,-2\}$ appears with a negative $a_p$-sign. Anni, B\"ockle, Gr\"af, and Troya also observed the high number of coincidences between slopes of $\L$-invariants appearing with $a_p = +p^{\frac{k}{2}-1}$ versus $a_p = -p^{\frac{k}{2}-1}$. See \cite[Section 6.3]{ABGT}, but note that the example of $k=28$ shows Conjecture 6.3(b) in {\em loc.\ cit.} is slightly misphrased. 

Synthesizing the observation, we propose two principles.
\begin{enumerate}[label=(\roman*)]
\item The slopes of $\L$-invariants, with multiplicity, occurring for eigenforms of each $a_p$-sign is more or less independent of the sign.\label{enum:sd-same-prin}
\item The slopes that are imbalanced between the two signs are those that are nearest to zero.\label{enum:sd-bias-prin}
\end{enumerate}
We systematically tested these principles  as follows. For each prime $p$ and each tame level $N$ we define
\begin{equation*}
\mathbf v_k^{\pm} = \{v_p(\L_f) \mid f \in S_k^{\pm}(\Gamma_0)\}
\end{equation*}
as a multi-set. For each $v$ equal to a slope of an $\L$-invariant we determined whether $v$ occurred with equal multiplicity in $\mathbf v_k^{+}$ versus $\mathbf v_k^{-}$. If so, we call $v$ sign independent in weight $k$, otherwise we call $v$ sign dependent.

To illustrate this, the sign dependent slopes in levels $\Gamma_0(2)$ and $\Gamma_0(7)$ are plotted in Figures \ref{fig:sign_bias_2_2_raw} and \ref{fig:sign_bias_7_7_raw}. In each figure, the upper plot is a scatter plot of {\em all} the sign {\em dependent} slopes. The bottom plot is a bar chart tabulating the total number of sign dependent slopes weight-by-weight. The (blue) triangles pointing up represent slopes with higher multiplicity in $\mathbf v_k^+$ than $\mathbf v_k^-$, while the downward pointing (red) triangles represent the opposite.

\begin{figure}[htbp]
\centering
\includegraphics[scale=.75]{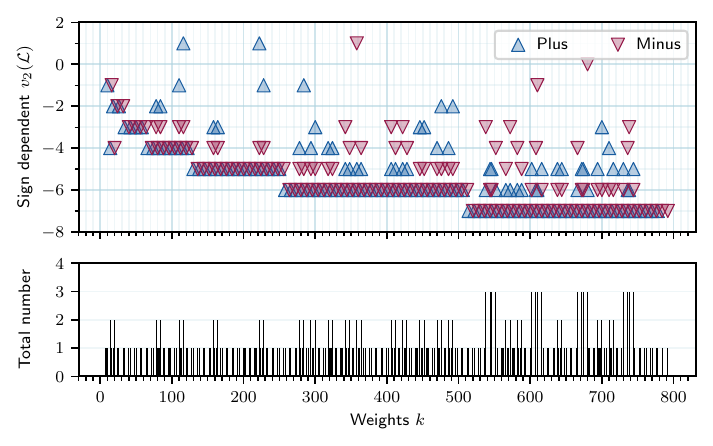}
\caption{Sign bias for $\L$-invariants in level $\Gamma_0(2)$ with respect to weights $10 \leq k \leq 730$. Upper: Scatter plot of sign dependent slopes.
Lower: Frequencies of sign dependent slopes.}
\label{fig:sign_bias_2_2_raw}
\end{figure}

\begin{figure}[htbp]
\centering
\includegraphics[scale=.75]{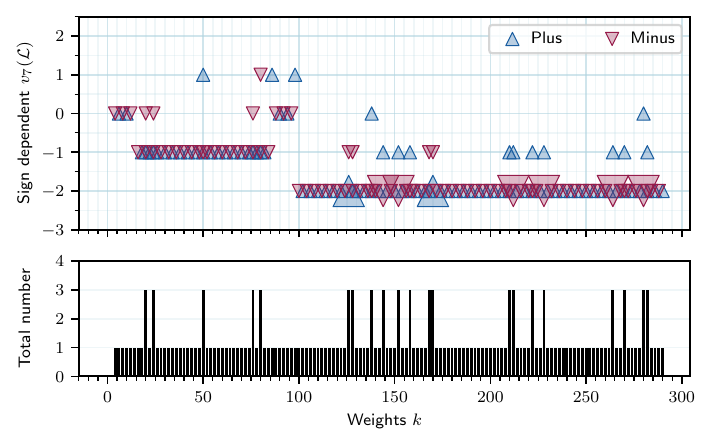}
\caption{Sign bias for $\L$-invariants in level $\Gamma_0(7)$ with respect to weights $4 \leq k \leq 290$. Upper: Scatter plot of sign dependent slopes.
Lower: Frequencies of sign dependent slopes.}
\label{fig:sign_bias_7_7_raw}
\end{figure}

Both plots support our two principles above. For instance, in level $\Gamma_0(2)$ the most negative sign dependent $\L$-invariant slope we found was $-7$. This is extremely close to zero, once we note the most negative $\L$-invariant slope over the same data range is less than $-120$ (see Figure \ref{fig:2_2_raw_huge}). It seems plausible to expect sign dependent slopes $v$ to satisfy $v \geq -c \log(k)$ where $c$ is a positive constant. (Depending on what? Perhaps just $p$.) It also seems reasonable that the {\em total number} of sign dependent slopes in weights $k \leq T$ is $O(T \log(T))$, which is a magnitude smaller than roughly $T^2$-many eigenforms of weight $k \leq T$.

As further evidence that sign dependent slopes are scarce, we provide Table \ref{table:sign-dependent-slopes}, which we explain now. In it, we  gather the total amount of sign dependent slopes in two stages. For each fixed $p$ and $N$, we determined the midpoint $k_{\mathrm{mid}}$ of the weights for which data was gathered. (For instance, in level $\Gamma_0(2)$ we gathered data for weights $10 \leq k \leq 730$ and so the midpoint is roughly $370$.) Then, we determined the percentage of sign dependent slopes occurring in weights $k\leq k_{\mathrm{mid}}$ versus all $k$. The first percentage occurs in the fifth column of Table \ref{table:sign-dependent-slopes} while the second is reported in the final column. The ratio of the percentages is roughly $2:1$, which is consistent with the number of sign dependent slopes in weights $k \leq T$ being a magnitude less than $T^2$.

\begin{table}[htp]
\renewcommand{\arraystretch}{1.1}
\begin{tabularx}{\textwidth}{|C|C|C|C|C|C|C|}
\hline $p$ & $Np$ & \# weights & \# data up $k_{\mathrm{mid}}$ & \% sign dependent up to $k_{\mathrm{mid}}$ & \# total data & \% sign dependent \\ \hline\hline
$2$ & $2$ & $391$ & $3434$ & $4.14$ & $13002$ & $2.37$ \\
 $2$ & $6$ & $165$ & $1291$ & $9.37$ & $4603$ & $5.02$ \\
 $2$ & $10$ & $136$ & $1680$ & $7.5$ & $6264$ & $3.64$ \\
 $2$ & $14$ & $126$ & $2245$ & $8.06$ & $7947$ & $3.49$ \\
 $2$ & $22$ & $103$ & $2340$ & $4.79$ & $9013$ & $2.34$ \\
 $3$ & $3$ & $252$ & $2730$ & $4.25$ & $10752$ & $2.18$ \\
 $3$ & $6$ & $164$ & $1291$ & $6.89$ & $4546$ & $4.4$ \\
 $3$ & $15$ & $72$ & $962$ & $12.47$ & $3594$ & $6.12$ \\
 $3$ & $21$ & $70$ & $1296$ & $9.57$ & $5040$ & $4.88$ \\
 $3$ & $33$ & $42$ & $735$ & $14.83$ & $2941$ & $7.17$ \\
 $5$ & $5$ & $243$ & $5371$ & $3.18$ & $20341$ & $1.61$ \\
 $5$ & $10$ & $135$ & $1680$ & $5.24$ & $6173$ & $2.64$ \\
 $5$ & $15$ & $72$ & $962$ & $8.94$ & $3594$ & $5.29$ \\
 $5$ & $35$ & $41$ & $883$ & $10.31$ & $3363$ & $6.87$ \\
 $5$ & $70$ & $20$ & $280$ & $9.29$ & $960$ & $5.62$ \\
 $7$ & $7$ & $144$ & $2664$ & $3.38$ & $10512$ & $1.77$ \\
 $7$ & $14$ & $90$ & $1301$ & $7.76$ & $4085$ & $4.38$ \\
 $7$ & $21$ & $70$ & $1296$ & $6.64$ & $5040$ & $3.37$ \\
 $7$ & $35$ & $41$ & $883$ & $10.76$ & $3363$ & $5.68$ \\
 $7$ & $70$ & $19$ & $280$ & $10.0$ & $874$ & $6.41$ \\
 $11$ & $11$ & $99$ & $2166$ & $5.08$ & $8332$ & $2.66$ \\
 $13$ & $13$ & $85$ & $1935$ & $3.05$ & $7433$ & $1.57$ \\
\hline
\end{tabularx}
\caption{Summary statistics for sign dependent slopes of {\em newforms} in level $\Gamma_0(Np)$ for various tame levels $N$ and primes $p$.}
\label{table:sign-dependent-slopes}
\end{table}

\subsection{Galois representations}\label{subsec:data-galois}
We use the following notations for Galois representations associated to an eigenform $f$. Its global $p$-adic representation is denoted by
$$
\rho_f : \Gal(\Qbar/\Q) \rightarrow \GL_2(\Qpbar).
$$
The semi-simplification modulo $p$ is denoted by $\rhobar_f$. 

Our method to extract $\L$-invariant data diagonalizes spaces of modular symbols as one of its steps. So, alongside $\L$-invariant data, we are able to tabulate eigensystems. In particular, for each eigenform $f$ we simultaneously determined both $v_p(\L_f)$ and the global mod $p$ Galois representation $\rhobar_f$. The main benefit is that we can filter the previous data according to one of the finitely many $\rhobar$ at level $\Gamma_0$.

We illustrate this first in a small example. Write $\omega$ for the cyclotomic character modulo $p$. In levels $\Gamma_0(2)$ and $\Gamma_0(3)$ the only $\rhobar$ is $1 \oplus 1$ and $1\oplus \omega$, respectively. Therefore, the data in Figures \ref{fig:2_2_raw_huge} and \ref{fig:3_3_raw} represent data for eigenforms with fixed mod $p$ Galois representation, already. In level $\Gamma_0(5)$, there are four distinct global representations:\ $1 \oplus \omega$ and its three non-trivial cyclotomic twists $\omega \oplus \omega^2$, $\omega^2 \oplus \omega^3$, and $1 \oplus \omega^3$. Thus in Figure \ref{fig:5_5_Galois_representations_apart} we re-plot the $\Gamma_0(5)$-data in four separate plots each corresponding to a fixed $\rhobar$, and in Figure \ref{fig:5_5_Galois_representations} we re-overlay the data but keep the color and shape codings intact.  (We also took the liberty to reduce the maximum weight in Figure \ref{fig:5_5_Galois_representations}, in order to make the plot easier to study ``by-eye''.) 

\begin{figure}[htbp]
\centering
\includegraphics[scale=.75]{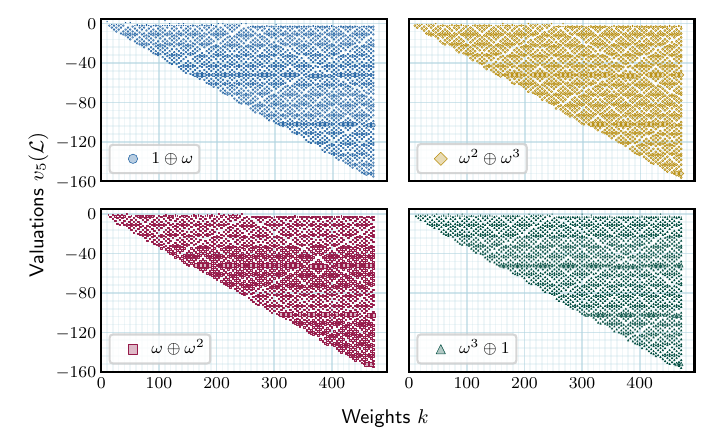}
\caption{Slopes of $\L$-invariants in level $\Gamma_0(5)$ with respect to weights $10 \leq k \leq 472$, filtered according to their Galois representations. (Total data:\ 18560 eigenforms.)}
\label{fig:5_5_Galois_representations_apart}
\end{figure}

\begin{figure}[htbp]
\centering
\includegraphics[scale=.75]{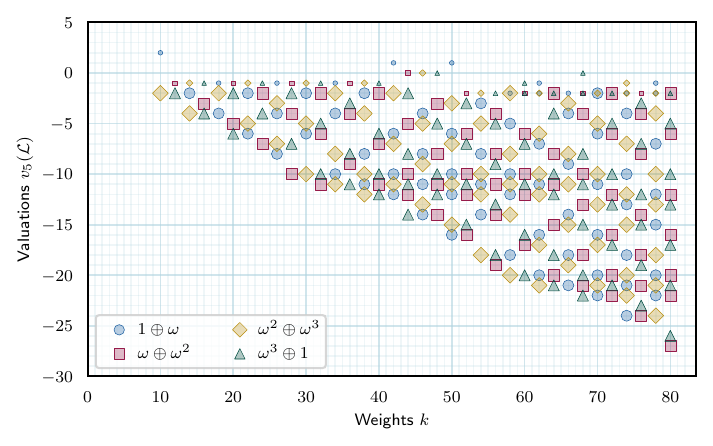}
\caption{Slopes of $\L$-invariants in level $\Gamma_0(5)$ with respect to weights $10 \leq k \leq 80$, filtered according to their Galois representations. (Total data:\ 287 eigenforms.)}
\label{fig:5_5_Galois_representations}
\end{figure}

The plots are rather astounding. To be clear, the literal data in the four plots of Figure \ref{fig:5_5_Galois_representations_apart} are not the same (as illustrated in Figure \ref{fig:5_5_Galois_representations}). After all, the data is not even  supported on the same weights because each $\rhobar$ lifts to only half even weights, depending on whether $k \bmod 4$. Even if you fix the congruence class $k\bmod 4$, the data are not the same. Rather, they interweave with one another while the plots for a fixed $\rhobar$ maintain the basic triangular shape and cellular structures we have seen in all our data thus far.

Figure \ref{fig:11_11_raw_highlight} offers a second plot to visualize how $\rhobar$-filtered data sits within an unfiltered dataset. Namely, let $\rhobar$ be the mod $11$ representation associated with the elliptic curve $X_0(11)$. Then we consider all eigenforms $f$ for which $\rhobar_f \simeq \rhobar \otimes \omega^j$ for some $j=0,1,\dotsc,10$. We call these eigenforms ``twists of $X_0(11)$ modulo $11$''. Each such $\rhobar_f$ is surjective, which makes the context slightly different than the Eisenstein examples above. 

Now, here is what we are showing in Figure \ref{fig:11_11_raw_highlight}.
\begin{enumerate}
\item The black dots represent data gathered from eigenforms that are twists of $X_0(11)$ modulo $11$, including with multiplicity.
\item The larger blue dots represent {\em all} the data gathered in level $\Gamma_0(11)$.
\end{enumerate}

\begin{figure}[htbp]
\centering
\includegraphics[scale=.75]{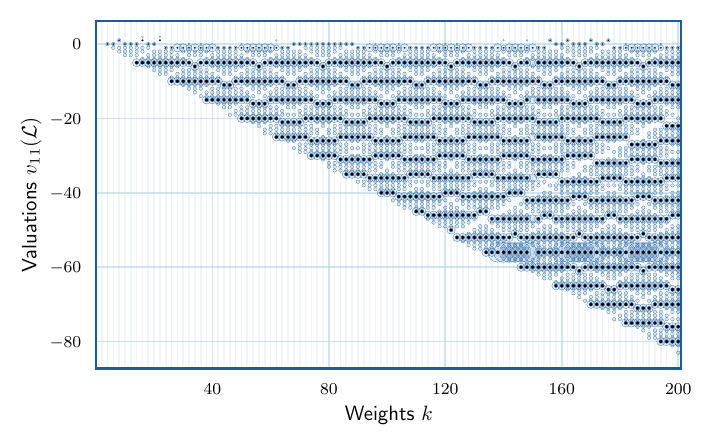}
\caption{$11$-adic slopes of $\L$-invariants in level $\Gamma_0(11)$ for eigenforms of weight $4 \leq k \leq 200$ underneath those eigenforms that lift a twist of $X_0(11)$. (Total data:\ 8332 eigenforms. $X_0(11)$ data:\ 1655 eigenforms.)}
\label{fig:11_11_raw_highlight}
\end{figure}

The plot is consistent with what we have already seen. First, there is no major qualitative difference between the total data at level $\Gamma_0(11)$ and the data for twists of $X_0(11)$ modulo $11$. The basic triangular shape and its somewhat cellular structure remains for the filtered data. If you look vertically above a weight $k$, you will, however, notice space where the full $\Gamma_0(11)$-data is supported, but not data for twists of $X_0(11)$ modulo $11$. We will make more sense of that in Section \ref{subsec:galois} below.

\section{Galois-theoretic perspectives}
\label{sec:theory}

This section is an interlude describing some underlying theories related to Galois representations. In the next section, we re-analyze the raw data in light of the theory described here.

We begin by recalling the Fontaine--Mazur definition of $\L$-invariants. Then, we discuss perspectives on local and global deformations of Galois representations.

\subsection{$\L$-invariants and local Galois representations}\label{subsec:Linvs-local-Galois}

Fix $k\geq 2$, a sign $\pm$, and $\L \in \Qpbar$. We write  $D = \Qpbar e_1 \oplus \Qpbar e_2$ and define $\Qpbar$-linear operators $\varphi$ and $N$ on $D$ by
\begin{align*}
\varphi &= \begin{pmatrix} \pm p^{\frac{k}{2}} & 0 \\ 0 & \pm p^{\frac{k}{2}-1} \end{pmatrix}\\
N &= \begin{pmatrix} 0 & 0 \\ 1 & 0 \end{pmatrix},
\end{align*}
in the ordered basis $(e_1,e_2)$.  We endow $D$ with a filtration $\Fil^{\bullet}D$ by
\begin{equation*}
\mathrm{Fil}^{j}D = \begin{cases}
D & \text{if $j\leq 0$;}\\
\Qpbar\langle e_1 + \L e_2 \rangle & \text{if $1 \leq j\leq k-1$;}\\
(0) & \text{otherwise.}
\end{cases}
\end{equation*}
In this way, $D$ is a two-dimensional $\Qpbar$-linear weakly-admissible filtered $(\varphi,N)$-module in the sense of Fontaine's $p$-adic Hodge theory of local Galois representations \cite{Fontaine-RepresentationSemiStable}. A theorem of Colmez and Fontaine (\cite[Th\'eorm\`eme A]{CF}) implies there exists a two-dimensional local representation
\begin{equation*}
r : \Gal(\Qpbar/\Qp) \rightarrow \GL_{\Qpbar}(V) \simeq \GL_2(\Qpbar)
\end{equation*}
that is semi-stable and non-crystalline, since $N\neq 0$ on $D$, for which $D_{\mathrm{st}}(r) = D$. The Hodge--Tate weights of $r$ equal $\{0,k-1\}$ and $\det(r) = \omega^{k-1}$.  We write $r_{k,\L}^{\pm}$ for this representation, or we write $V_{k,\L}^{\pm}$ when it is more convenient to reference the vector space upon which the Galois groups acts. Any two-dimensional, semi-stable, and non-crystalline, $\Qpbar$-linear  $\Gal(\Qpbar/\Qp)$-representation with distinct Hodge--Tate weights is isomorphic to a twist of some such $r_{k,\L}^{\pm}$. The representations $r_{k,\L}^{\pm}$ are irreducible if and only if $k > 2$. See \cite[Exemple 3.1.2.2]{BM} for a similar summary and references.

The connection with eigenforms is as follows. Let $f$ be a $\Gamma_0$-eigenform. Write 
\begin{equation*}
\rho_f : \Gal(\Qbar/\Q) \rightarrow \GL_2(\Qpbar)
\end{equation*}
for its global $p$-adic Galois representation, as before. Now, set
\begin{equation*}
r_f : \Gal(\Qpbar/\Qp) \rightarrow \GL_2(\Qpbar)
\end{equation*}
for the corresponding local restriction at $p$. A foundational theorem of Saito \cite{Saito} implies that
\begin{equation*}
r_f \simeq r_{k,\L_f}^{\pm},
\end{equation*}
where the sign $\pm$ is given by the sign of $a_p(f) = \pm p^{\frac{k}{2}-1}$. (This is the {\em definition} of the Fontaine--Mazur $\L$-invariant for the eigenform $f$, except Mazur originally chose the negative of our convention.)

In particular, from this perspective, the $\L$-invariant can be conceptualized as follows. When $f$ is a $\Gamma_0$-eigenform, its local Galois representation $r_f$ is completely determined by discrete data $(k,\pm)$ and a {\em continuous} parameter $\L$. 

To avoid complications below, it is convenient here to include a definition of $r_{k,\L}^{\pm}$ for $\L = \infty$. Namely, $r = r_{k,\infty}^{\pm}$ is the unique two-dimensional, irreducible, {\em crystalline} representation of $\Gal(\Qpbar/\Qp)$ such that $D_{\mathrm{crys}}(r) = \Qpbar e_1 \oplus \Qpbar e_2$ is the weakly-admissible filtered $(\varphi,N)$-module, with $N = 0$, and for which
\begin{equation*}
\varphi = \begin{pmatrix} \pm p^{\frac{k}{2}} & 0 \\ 0 & \pm p^{\frac{k}{2}-1}\end{pmatrix} \;\;\;\;\;\; \Fil^j D_{\mathrm{crys}}(r) = \begin{cases}
D_{\mathrm{crys}}(r) & \text{if $j\leq 0$;}\\
\Qpbar\langle e_1 + e_2 \rangle & \text{if $1 \leq j \leq k-1$;}\\
(0) & \text{otherwise}.
\end{cases}
\end{equation*}
This representation is irreducible if and only if $k > 2$, just like the $r_{k,\L}^{\pm}$ with $\L \neq \infty$ above.

Note that it is impossible that there exists an eigenform $f$ (of any level) such that $r_f \simeq r_{k,\infty}^{\pm}$. Indeed, Saito's theorem would then imply $f$ has  level prime-to-$p$ and
\begin{equation*}
a_p(f) = \tr(\varphi|_{D_{\mathrm{st}}(r)}) = \pm(p^{\frac{k}{2}} + p^{\frac{k}{2}-1}),
\end{equation*}
which violates the Ramanujan--Petersson bound.

\subsection{Loci with fixed reductions}\label{subsec:rbar-loci}
Here we push the idea of $\L$ as a continuous parameter further. Now consider a local mod $p$ representation
\begin{equation*}
\rbar : \Gal(\Qpbar/\Qp) \rightarrow \GL_{\Fpbar}(\overline V) \simeq \GL_2(\Fpbar).
\end{equation*}
Given $k$ and $\pm$, we write
\begin{equation*}
X_{k}^{\pm}(\rbar) =  \left\{ \L \in \P^1(\Qpbar) ~\Big|~ \substack{\ds   \text{~there exists a $\Zpbar$-linear $\Gal(\Qpbar/\Qp)$-stable~}  \\ \ds \text{~lattice $T \subseteq V_{k,\L}^{\pm}$ such that $T/\mathfrak m_{\Zpbar}T \simeq \overline V$} }\right\}.
\end{equation*}
The sets $X_k^{\pm}(\rbar)$ are closely related to deformation spaces of Galois representations, and they can be described quite concretely. 

To start, for any representation $V$ of a group, let $V^{\mathrm{ss}}$ denotes its semi-simplification. If $\rbar$ has scalar endomorphisms, we  define $\rbar_0 := \rbar$. Otherwise, assume $\rbar = \chi_1 \oplus \chi_2$ with $\chi_1 \neq \chi_2$. Without loss of generality, we insist that
\begin{equation*}
\dim_{\Fpbar} H^1(\Gal(\Qpbar/\Qp),\Fpbar(\chi_1\chi_2^{-1})) = 1.
\end{equation*}
Given $\chi_1\neq \chi_2$, the order of $\chi_1$ and $\chi_2$ matters only if $\chi_1\chi_2^{-1} = \omega^{\pm 1}$. Then, we define $\rbar_0$ as the reducible, but non-split, representation
\begin{equation*}
\rbar_0 \simeq \begin{pmatrix} \chi_1 & \ast \neq 0 \\ 0 & \chi_2 \end{pmatrix}
\end{equation*}
representing the unique-up-to-scalar class in $H^1$. Therefore, in all cases, $\rbar_0$ has scalar endomorphisms and $\rbar_0^{\mathrm{ss}} = \rbar^{\mathrm{ss}}$.

\begin{prop}\label{prop:ribets-lemma}
Assume $k > 2$ and $\rbar$ is not scalar, so that $\rbar_0$ is well-defined. Then, $X_k^{\pm}(\rbar) = X_k^{\pm}(\rbar_0)$.
\end{prop}
\begin{proof}
If $\rbar = \rbar_0$ there is nothing to prove. Otherwise, we assume $\rbar = \chi_1 \oplus \chi_2$ for distinct characters $\chi_j$, ordered as above. The rest of the proof follows from a famous observation of Ribet \cite[Section 2]{Ribet-Invent76}.

Suppose that $\L \in X_k^{\pm}(\rbar)$. Since $k > 2$, the representation $V_{k,\L}^{\pm}$ is irreducible and defined over some finite extension $E \supseteq \Qp$. By Ribet's \cite[Proposition 2.1]{Ribet-Invent76} there exists $\mathcal O_E$-linear and $\Gal(\Qpbar/\Qp)$-stable lattice $\Lambda \subseteq V_{k,\L}^{\pm}$ for which 
\begin{equation}\label{eqn:lattice-Lambda}
\Lambda/\mathfrak m_{\mathcal O_E}\Lambda \simeq \begin{pmatrix} \chi_1 & \ast \neq 0 \\ 0 & \chi_2 \end{pmatrix}.
\end{equation}
We have chosen the $\chi_j$ so that there is only one such representation, and thus $\Lambda/\mathfrak m_{\mathcal O_E}\Lambda \simeq \rbar_0$. Therefore, we have $X_{k}^{\pm}(\rbar) \subseteq X_{k}^{\pm}(\rbar_0)$.

Now suppose $\L \in X_k^{\pm}(\rbar_0)$. Then, there is a lattice $\Lambda \subseteq V_{k,\L}^{\pm}$ as in \eqref{eqn:lattice-Lambda}. After replacing $E$ by $E(\sqrt{p})$, if necessary, we can find a new $\mathcal O_E$-linear and Galois-stable lattice $T$ that contains $\Lambda$ and for which $T/\mathfrak m_{\mathcal O_E}T$ is semi-simple as a $\Gal(\Qpbar/\Qp)$-representation. In particular, $T/\mathfrak m_{\mathcal O_E}T \simeq \rbar$ and thus we have $X_{k}^{\pm}(\rbar_0) \subseteq X_{k}^{\pm}(\rbar)$.
\end{proof}

Now suppose that $\rbar$ has only scalar endomorphisms. In this case, there is a universal deformation ring $R_{\rbar}$, going back to Mazur \cite{Mazur-Deformations}. We will replace this ring with a certain quotient we call $R_{\rbar}^{k,\pm}$. 

To define $R_{\rbar}^{k,\pm}$, Kisin first constructed (see \cite{Kisin-PST}), for each $k$, a natural quotient $R_{\rbar} \twoheadrightarrow R_{\rbar}^{k,\mathrm{st}}$ that parametrizes semi-stable deformations of $\rbar$, with Hodge--Tate weights $0 < k-1$ and determinant $\omega^{k-1}$. The semi-stable condition is partially encoded in Kisin's ring by means of the inertial type. For comparison, we are taking the trivial inertial type. Let $\mathrm X_{\rbar}^{k,\mathrm{st}} = \Spf(R_{\rbar}^{k,\mathrm{st}})^{\mathrm{rig}}$ be the rigid generic fiber. The rigid space $\mathrm X_{\rbar}^{k,\mathrm{st}}$ is always reduced and equidimensional. Conceptually, $\mathrm X_{\rbar}^{k,\mathrm{st}}$ breaks up into a union of the crystalline locus and two loci that contain semi-stable and non-crystalline representations with one of the two signs $\pm$. For most $\rbar$, these three loci do not intersect. For $k > 2$, the only case where intersection is possible is when $\rbar \simeq \mathrm{ind}(\omega_2^{k-1})$. In that case, the issue is that, for either $\pm$, we have
\begin{equation*}
\rbar_{k,\infty}^{\pm} \simeq \mathrm{ind}(\omega_2^{k-1}),
\end{equation*}
and there are three smooth irreducible components meeting at $r_{k,\infty}^{\pm} \in \mathrm{X}_{\rbar}^{k,\mathrm{st}}(\Qpbar)$.

The ring $R_{\rbar}^{k,\pm}$ picks out the semi-stable and non-crystalline deformations with a fixed sign $\pm$. To define it, one relies on extending the trivial inertial type to a Galois type, in the terminology of Rozensztajn \cite[Definition 5.1.1(3)]{Roz}. Specifically, let $W_{\Qp}$ be the Weil group and denote $\tau_{\pm} = \unr(\pm 1) \oplus \unr(\pm p)$ (a smooth $W_{\Qp}$-representation). Then, $R_{\rbar}^{k,\pm}$ is defined as the maximal reduced quotient of $R_{\rbar}^{k,\mathrm{st}}$ supported on the union of components in $\Spec(R_{\rbar}^{k,\mathrm{st}})$ that contain at least one semi-stable lift $r$ of $\rbar$ such that the associated Weil--Deligne representation $\mathrm{WD}(r)$ has non-trivial monodromy and Frobenius acting by $\tau_{\pm}$.  See \cite[Section 2.3.3]{Rozensztajn-Pst}. 

We then define $\mathrm X_{\rbar}^{k,\pm} = \Spf(R_{\rbar}^{k,\pm})^{\mathrm{rig}}$ to be the rigid space associated with $R_{\rbar}^{k,\pm}$. Each $\mathrm X_{\rbar}^{k,\pm}$ is a union of components in $X_{\rbar}^{k,\mathrm{st}}$. The rigid space $\mathrm X_{\rbar}^{k,\pm}$ is smooth. The function $\L \mapsto r_{k,\L}^{\pm}$ defines a bijection 
\begin{equation*}
\L : X_{k}^{\pm}(\rbar) \rightarrow \mathrm{X}_{\rbar}^{k,\pm}(\Qpbar)
\end{equation*}
by \cite[Theorem 5.3.1]{Roz} and Section 7.6 of {\em op.\ cit.}.  Rozensztajn also proves in \cite[Corollary 5.3.2]{Roz} that this makes $X_{k}^{\pm}(\rbar)$ a {\em standard} subset of $\P^1(\Qpbar)$. This means it is built as follows. You begin with either $\P^1(\Qpbar)$ or a disc $v_p(\L - x) > r$ defined by an open inequality. From that, you remove a finite union of discs $v_p(\L - y) \geq s$ defined by closed inequalities. The outcome is called a {\em connected} standard subset. A finite union of connected standard subsets is called a standard subset (\cite[Definition 3.3.1-3.3.2]{Roz}). 

Based on this discussion, we have two (equivalent) notions of a {\em component} of $X_k^{\pm}(\rbar)$. For one, we can take $\L(U(\Qpbar))$ where $\L$ is the rigid analytic map above and $U \subseteq \mathrm X_{\rbar}^{k,\pm}$ is a rigid analytic component. A second would be that a {\em component} is a connected standard subset in $X_{k}^{\pm}(\rbar)$. The latter definition also makes sense when $\rbar$ has non-scalar endomorphisms, by Proposition \ref{prop:ribets-lemma} (at least if $k > 2$ and $\rbar$ is non-scalar).

\subsection{A local-global principle}
\label{sec:global}

The prior subsections concerned purely local questions on Galois representations. The data gathered in the research underlying this article is global. So, here, we recall how to perceive the $\L$-invariant data for  $\Gamma_0$-eigenforms as a sampling of the components on $X_k^{\pm}(\rbar)$.

Fix
\begin{equation*}
\rhobar : \Gal(\Qbar/\Q) \rightarrow \GL_2(\Fpbar)
\end{equation*}
that is modular of level $\Gamma_0(N)$. Similar to Section \ref{subsec:data-galois}, we will focus on eigenforms $f$ for which $\rhobar_f \simeq \rhobar$. Namely, define $\rbar = \rhobar|_{\Gal(\Qpbar/\Qp)}$, and then define 
\begin{equation*}
X_k^{\pm}(\rhobar) \subseteq X_k^{\pm}(\rbar)
\end{equation*}
as the set of $\L$ such that there exists a $\Gamma_0$-eigenform $f$ of weight $k$ for which $\rhobar_f \simeq \rhobar$ and $\L = \L_f$. (It is automatic then that $a_p(f) = \pm p^{\frac{k}{2}-1}$.) The {\em finite} set of points $X_k^{\pm}(\rhobar)$ is referred to as the {\em modular} points on $X_k^{\pm}(\rbar)$. (Note that this implicitly assumes we have fixed the level $N$ and $\rhobar$ to begin with.)

In situations where modularity lifting theorems are proven via patching methods of Taylor--Wiles and Kisin, we have the following theorem on modular points realizing local components.

\begin{thm}[Kisin]\label{thm:local-global}
Suppose $p > 2$ and $\rhobar|_{\Gal(\Qbar/\Q(\zeta_p))}$ is irreducible, and suppose that $\rbar$ has only scalar endomorphisms. Then, every component of $X_k^{\pm}(\rbar)$ contains a modular point.
\end{thm}

The proof of the theorem is based on Kisin's approach to proving the Fontaine--Mazur conjecture via the Breuil--M\'ezard conjecture \cite{Kisin-FM} and is implicit in the proof of \cite[Proposition 3.7]{Calegari-EvenGalois-2}. 
Under the assumptions on $\rhobar$ and $\rbar$, the Breuil--Me\'zard conjecture for $R_{\rbar}^{k,\mathrm{st}}$ was proven by Kisin and Paskunas \cite{Kisin-FM,Paskunas-BM}. Then, the argument in \cite[Proposition 3.7]{Calegari-EvenGalois-2} shows that every component of $\mathrm X_{\rbar}^{k,\mathrm{st}}$ contains a modular point. Since $\mathrm X_{\rbar}^{k,\pm}$ is a union of components in $\mathrm X_{\rbar}^{k,\mathrm{st}}$, Theorem \ref{thm:local-global} follows.
(We learned this argument from the crystalline version of Theorem \ref{thm:local-global} in  \cite[Proposition 5.1]{BG-Survey}.)

 The analysis of Theorem \ref{thm:local-global} relies on the deformation ring $R_{\rbar}^{k,\pm}$. Regardless of $\rbar$, but still making the assumptions on $\rhobar$, the logic above shows that each component of the framed deformation ring $R_{\rbar}^{k,\pm,\square}$ supports a modular point, but it is not clear how to connect such components to what we have defined to be $X_{k}^{\pm}(\rbar)$ (since it is not completely clear how to relate components of $X_{k}^{\pm}(\rbar)$ with those of $\mathrm{Spf}(R_{\rbar}^{k,\pm,\square})^{\mathrm{rig}}$). So, when $\rbar = \chi_1 \oplus \chi_2$, it is not completely clear what to say.

To summarize the impact of Theorem \ref{thm:local-global}, it means that when we look at the scatter plots, their structure is being {\em constrained} in a fairly faithful way by the purely local spaces $X_k^{\pm}(\rbar)$. In the situations where the assumptions on $\rhobar$ are satisfied, the data is a representation of the location of the components in $X_k^{\pm}(\rbar)$ with respect to the $\L$-coordinate on $X_k^{\pm}(\rbar)$. Note, we are not always in the irreducible case (especially for small primes and small tame levels) but in any case the global data is sampling some collection of the components.

\section{The data and Galois representations}
\label{sec:reanalysis}

The goal of this section is to re-discuss the raw data within the framework of Galois representations. What properties definitely reflect theory? Which might we {\em hope} to develop into reflections of theory? And which are perplexing? Answers to these questions are interspersed throughout the discussion that follows.

\subsection{Thresholds}\label{subsec:thresholds-galois}

Here, we revisit the thresholds previously described in Section \ref{subsec:thresholds}. Recall, the data seems to indicate that as $k \rightarrow +\infty$, for all but a thin set of eigenforms, we have
\begin{equation}\label{eqn:threshold-bounds}
-\frac{(p-1)k}{2(p+1)} \leq v_p(\L_f) \leq 0.
\end{equation}
We cannot prove this at the moment, but we can use Galois representations to explain why {\em some} kind of bounds should exist. The main benefit is that it shifts the primary question away from global arithmetic and toward local arithmetic.

First, we can justify why the data in Figure \ref{fig:2_2_raw_huge} and Figures \ref{fig:3_3_raw}-\ref{fig:11_11_raw} all satisfy some lower linear threshold. For terminology, among all
\begin{equation*}
\rbar : \Gal(\Qpbar/\Qp) \rightarrow \GL_2(\Fpbar),
\end{equation*}
the {\em irregular} ones are those of the form $\rbar = \ind(\omega_2^s)$ with $s \in \Z$. All other representations $\rbar$ are called regular. When considering eigenforms of even weight, we have an exact equivalence that $\rbar_f$ is irregular if and only if $\rbar_f$ is irreducible. In that case we call the eigenform $f$ irregular and otherwise call the eigenform regular. These definition are given in \cite{BP-Ghost2,BP-Ghost3}, building on a definition by Buzzard related to eigenforms \cite{Buzzard-SlopeQuestions}. The recent breakthrough of Liu, Truong, Xiao, and Zhao \cite{LTXZ-Ghost,LTXZ-Proof} is focused on regular eigenforms (of certain generic weights modulo  $p$). Here is a theorem on $\L$-invariants of regular eigenforms, based on a theorem of Levin, Liu, and the first author.
\begin{thm}[{\cite{BLL}}]\label{thm:BLL-intext}
If $f \in S_k(\Gamma_0)$ is a regular eigenform, then 
\begin{equation}\label{eqn:threshold-BLL}
 -\frac{(p+1)k}{2(p-1)} + \frac{p+3}{2(p-1)} \leq v_p(\L_f).
\end{equation}
\end{thm}

It is a basic fact (and a finite computation) that if $p < 59$, then every eigenform of level $\Gamma_0(p)$ is regular. Therefore, Theorem \ref{thm:BLL-intext} implies that the data in Figure \ref{fig:2_2_raw_huge} and Figures \ref{fig:3_3_raw}-\ref{fig:11_11_raw} must all, {\em a priori}, lie above a linearly decreasing threshold. 

However, the threshold provided by Theorem \ref{thm:BLL-intext} is {\em not} the same as the one predicted in \eqref{eqn:threshold-bounds}. The numerators and denominators are in fact nearly swapped between Theorem \ref{thm:BLL-intext} and \eqref{eqn:threshold-bounds}. The discrepancy is not surprising as the method to prove Theorem \ref{thm:BLL-intext} is likely suboptimal. Indeed, it is a corollary of \cite[Theorem 1.1]{BLL}, which is the purely local theorem that
\begin{equation*}
\rbar_{k,\L}^{\pm} = \rbar_{k,\infty}^{\pm} \simeq \ind(\omega_2^{k-1})
\end{equation*}
whenever
\begin{equation}\label{eqn:BLL-local-bound}
v_p(\L) < 2 - \frac{k}{2} - v_p((k-2)!).
\end{equation}
The argument of {\em loc.\ cit.}\ is based on analysis in $p$-adic Hodge theory, and theorems (\cite{BergdallLevin, BLZ}) proven via a similar $p$-adic Hodge theory approach, in the crystalline cases, also produce bounds that are a factor away from realities. (Tight bounds, in crystalline cases and for regular eigenforms are proven in \cite[Theorem 1.12]{LTXZ-Proof}. For a purely local approach to tight bounds, using $p$-adic local Langlands, see \cite[Theorem M]{Arsovski}.)	

Let us summarize the discussion in terms of Section \ref{subsec:rbar-loci}. First, $\rbar_{k,\infty}^{\pm} = \ind(\omega_2^{k-1})$, and therefore $X_k^{\pm}(\ind(\omega_2^{k-1}))$ always contains a component $X_\infty$ around $\L = \infty$. The theorem in \cite{BLL} is a result proving $X_\infty$ contains a disc of some explicit radius. This implies a bound like \eqref{eqn:threshold-BLL} for regular eigenforms, and it makes improving the result of {\em op.\ cit.}\ an interesting project. 

An equally interesting project would be determining whether this framework {\em also} provides evidence toward \eqref{eqn:threshold-bounds} for irregular eigenforms. Note that any eigenform giving a modular point on $X_\infty$ will be an irregular eigenform, and Theorem \ref{thm:local-global} forces such modular points to exist (up to global assumptions). How many points might there be in a fixed global context?  This seems like a tricky question. According to experts we wrote to and spoke with, a version of Theorem \ref{thm:local-global} that quantifies the number of modular points would require new ideas and analysis. Nonetheless, a very weak bound, such as a claim that $X_\infty$ contains at most a constant multiple of $\log(k)$-many modular points, would be enough to imply that Theorem \ref{thm:BLL-intext} holds for 100\% of the eigenforms, as the weight $k \rightarrow +\infty$. (This conclusion would follow also from the distribution conjecture in Section \ref{conj}.)

 And why stop there? Suppose
\begin{equation*}
\rhobar : \Gal(\Qbar/\Q) \rightarrow \GL_2(\Fpbar)
\end{equation*}
is a fixed Galois representation satisfying the assumptions in Theorem \ref{thm:local-global}. For a component $X \subseteq X_{k}^{\pm}(\rbar)$, let $X(\rhobar) = X \cap X_k^{\pm}(\rhobar)$. So, $X(\rhobar)$ is the finite set of $\L$-invariants coming from a global modular weight $k$ lift of $\rhobar$ (with a fixed sign $\pm$) and which belong to the specified local component $X$, i.e.\ $|X(\rhobar)|$ is the number of modular points on $X$. Then, might we expect there is a constant $c = c(\rhobar)$ such that for all $k$ and any component $X \subseteq X_k^{\pm}(\rbar)$,
\begin{equation}\label{eqn:component-bound}
|X(\rhobar)| \leq c \log(k)?
\end{equation}
A stricter bound such as $|X(\rhobar)| \leq c$ would be better from a statistical perspective and obviously would be the best case scenario. But, we have actually seen data and computations of Rozensztajn (private communication) that leads us to believe $|X(\rhobar)|$ is not bounded ranging over all $k$ and all $X$.

To end, we now return to the upper portion of the threshold \eqref{eqn:threshold-bounds}, for we can make sense of it in light of the hope \eqref{eqn:component-bound} and the preceding discussion. Indeed, for each even weight $k$, and each sign $\pm$, we can consider the $p$-adic representation $r_{k,0}^{\pm}$ occurring at $\L=0$. We do not have an explicit formula for $\rbar_{k,0}^{\pm}$, even up to semi-simplification. Nevertheless, there is {\em some} component $X_0 \subseteq X_k^{\pm}(\rbar_{k,0}^{\pm})$ containing $\L = 0$ and that component naturally contains an open disc $D_0 \subseteq X_0$. Therefore, some upper threshold in the style of \eqref{eqn:threshold-bounds} would follow from effective estimates for $D_0$, in the fashion of \cite{BLL}, along with the component-by-component bounds \eqref{eqn:component-bound} at a global level.

Determining the disc $D_0$ presents an interesting challenge. The Breuil--M\'ezard calculations \cite{BM} already show that $k \mapsto \rbar_{k,0}^{\pm}$ is not a simple function of $k$, in contrast to $k \mapsto \rbar_{k,\infty}^{\pm} \simeq \ind(\omega_2^{k-1})$. It is not transparent how to adjust the strategy in \cite{BLL} to hit that kind of moving target. Perhaps one might consider the recent work of Chitrao, Ghate, and Yasuda \cite{CitraoGhateYasuda-Limit} on comparing reductions modulo $p$ of crystalline representations to semi-stable ones as an alternative point of a view. In any case, we end this discussion on thresholds by just saying that the global data suggests that perhaps $D_0$ already contains a disc $v_p(\L) > c_0^{\pm}(k)$ where $c_0^{\pm}(k) \approx -\frac{\log(k)}{\log(p)}$ --- an entire hemisphere of $\L$-invariants in $\P^1$ seems to have constant reduction modulo $p$.

\subsection{Sign dependence and Galois representations}

We now turn to the sign (in)dependence of slopes of $\L$-invariants (as in Section \ref{subsec:signs}). First, we make a positive observation in favor of sign independence. Then, we do the opposite, explaining how to construct situations that {\em must} have sign dependent slopes of $\L$-invariants.

The positive observation is the following.

\begin{prop}\label{prop:irred-sign-independence}
If $\rbar$ is irregular, then $X^+_k(\rbar) = X^-_k(\rbar)$ for all $k\geq 2$.
\end{prop}
\begin{proof}
Let $\unr(a)$ denote the unramified character taking value $a$. The first point is that if $\rbar$ is irregular, then $\rbar \simeq \rbar \otimes \unr(-1)$. The second point is that, given any $\L \in \Qpbar$, we also have
\begin{equation*}
r_{k,\L}^+ \simeq r_{k,\L}^- \otimes \unr(-1).
\end{equation*}
This follows from the definitions outlined in Section \ref{subsec:Linvs-local-Galois} along with basic properties of calculations in $p$-adic Hodge theory. From these observations, it follows that $\L \in X_k^+(\rbar)$ if and only if $\L \in X_k^-(\rbar)$.
\end{proof}

Proposition \ref{prop:irred-sign-independence} does not {\em itself} have global implications to eigenforms, but you can still say something. Assume $\rhobar$ is a global mod $p$ representation and $f$ is an irregular $\Gamma_0$-eigenform with $\rhobar_f = \rhobar$. Assume as well that $p > 2$ and $\rhobar|_{\Gal(\Qbar/\Q(\zeta_p))}$ is irreducible. Write $a_p(f) = \pm p^{\frac{k}{2}-1}$. Thus $\L_f \in X_{k}^{\pm}(\rbar) = X_k^{\mp}(\rbar)$ (by Proposition \ref{prop:irred-sign-independence}). So, by Theorem \ref{thm:local-global}, there exists {\em another} $\Gamma_0$-eigenform $g$ with the following properties:
\begin{enumerate}[label=(\roman*)]
\item We have an equality of Galois representations $\rhobar_g = \rhobar = \rhobar_f$.
\item The sign of $a_p$ is $a_p(g) = \mp p^{\frac{k-2}{2}}$.
\item The $\L$-invariants $\L_g$ and $\L_f$ lie on a common component of $X_k^+(\rbar)=X_k^-(\rbar)$.
\end{enumerate}
In particular, if $\L_f$ happens to lie on a component of $X_{k}^{\pm}(\rbar)$ for which $v_p(-)$ is constant, then 
\begin{equation*}
v_p(\L_f) = v_p(\L_g),
\end{equation*}
 producing $\L$-invariants with sign independent slopes. (An example of such a component would be a rational disc containing neither $0$ or not $\infty$.)

In contrast with Section \ref{subsec:thresholds-galois}, this meta-argument here is focused on irregular eigenforms as opposed to regular ones, the ones that our data most commonly represent. And, it certainly does not explain why the sign independence for $\L$-invariants occurs even up to multiplicity.

Now we turn to constructing situations where $\L$-invariants {\em must} exhibit sign dependence. We introduce the following notation. Let $\rhobar$ be a fixed Galois representation modulo $p$. Define 
\begin{equation*}
d_k^{\pm}(\rhobar) = \# \left\{ f ~\Big|~ \ds f  \text{~a~$p$-new~eigenform~in~}  S_k(\Gamma_0) \text{~with $a_p(f) = \pm p^{\frac{k}{2}-1}$~and~$\rhobar_f \simeq \rhobar$} \right\}.
\end{equation*}
Thus, $d_k^{\pm}(\rhobar)$ is equal to the number of eigenforms lifting $\rhobar$ in weight $k$, with a fixed $\pm$-sign. The next lemma clarifies whether or not there exists such eigenforms in weight $k = 2$.

\begin{lemma}\label{lemma:sign-disc-2}
Assume $p > 2$. Suppose $\rhobar$ is modular of level $\Gamma_0(N)$. Then, there is at most one choice of $\pm$ for which $d_2^{\pm}(\rhobar) \neq 0$. More specifically, $d_2^{\pm}(\rhobar) \neq 0$ if and only if 
\begin{equation}\label{eqn:rhobar-local-descr}
\rhobar|_{\Gal(\Qpbar/\Qp)} \simeq \begin{pmatrix} 
\omega & \ast \\ 0 & 1
\end{pmatrix} \otimes \unr(\pm 1).
\end{equation}
\end{lemma}
\begin{proof}
First, if $f$ is a $p$-new eigenform in $S_2(\Gamma_0)$, then $f$ is {\em ordinary} at $p$. Deligne proved in the 1970's that, when $f$ is ordinary of weight $k$, we have
\begin{equation*}
\rhobar_f|_{\Gal(\Qpbar/\Qp)} \simeq \begin{pmatrix} 
\omega^{k-1} & \ast \\ 0 & 1
\end{pmatrix} \otimes \unr(a_p(f)).
\end{equation*}
(See \cite[Theorem 2]{Wiles-OrdinaryModular}.) Therefore, setting $k = 2$ and $a_p(f) = \pm p^{\frac{k}{2}-1} = \pm 1$, we see that if $d_2^{\pm}(\rhobar)$ is non-zero, then $\rhobar$ has the claimed shape. We have proven that there is at most one sign $\pm$ for which $d_2^{\pm}(\rhobar) \neq 0$, and we have identified the action of $\Gal(\Qpbar/\Qp)$ in that case. 

To finish the proof, suppose that $\rhobar$ has the shape given in \eqref{eqn:rhobar-local-descr}. We will show $d_2^{\pm}(\rhobar) \neq 0$. Since $\rhobar$ is modular of level $\Gamma_0(N)$, the weight part of Serre's conjecture offers two possibilities depending on the extension class $\ast$. We argue with either.
\begin{enumerate}[label=(Case \roman*)]
\item Suppose $\ast$ is a tr\`es ramif\'ee extension. Then, $\rhobar$ lifts to a level $\Gamma_0(N)$-eigenform $f$ of weight $p+1$ but {\em not} to any level $\Gamma_0(N)$-eigenform of weight $2$. 

Since $\rbar_f$ is reducible, the eigenform $f$ is ordinary. Therefore, it lives in a $p$-adic Hida family of $\Gamma_0$-eigenforms with weights satisfying $k \equiv 2 \bmod p-1$ and all lifting $\rhobar$. The fiber at $k = 2$ in this family must contain a classical eigenform that is {\em new} at $p$, since $\rhobar$ does not lift to level $N$ in weight $2$. Thus $d_2^{\pm}(\rhobar)$ is non-zero for some $\pm$, and we can tell which by looking at $\rhobar$.
\item Now suppose $\ast$ is a peu ramif\'ee extension. Then, there exists an eigenform $f$ of weight $2$ and level $\Gamma_0(N)$ for which $\rhobar_f = \rhobar$. 

In this case, we see $a_p(f) = \pm 1 \bmod p$, By Ribet's level-raising theorem (see \cite[Theorem 1, point 3.]{Ribet-RaisingTheLevels}), we conclude that there exists a {\em $p$-new} eigenform lifting $\rhobar$, which witnesses $d_{2}^{\pm}(\rhobar) \neq 0$.\qedhere
\end{enumerate}
\end{proof}
The argument of the prior lemma in which we showed there is  at most one sign for which $d_2^{\pm}(\rhobar)$ is non-zero can also be proven purely locally. In fact, the spaces $X_2^{\pm}(\rbar)$ can be read off from \cite[Proposition 4.2.1.1]{BM}. In such a reading, one can see immediately how the $\L$-invariant interacts with whether or not $\rbar$ is peu or tr\'es ramif\'ee.

Lemma \ref{lemma:sign-disc-2} gives a {\em sufficient} condition for there to be sign dependent $\L$-invariants slopes in weight $k = 2$. Namely, if $\rhobar$ satisfies the assumptions in that lemma, then there are sign dependent $\L$-invariant slopes because in fact $X_2^{\pm}(\rhobar)$ is empty for one choice of sign but not both.

To handle questions of higher weight, we introduce the notation
\begin{equation*}
\Delta_k(\rhobar) = d_k^{+}(\rhobar) - d_k^-(\rhobar).
\end{equation*}
Clearly, if $\Delta_k(\rhobar)$ is non-zero, then there {\em must} exist sign dependent slopes of $\L$-invariants for $\Gamma_0$-eigenforms of weight $k$. The next theorem of Anni, Ghitza, and Medvedovsky reduces whether or not $\Delta_k(\rhobar)$ is non-zero in weight $k$, to the same question for twists of $\rhobar$ in smaller weights.

\begin{thm}[{\cite{{AGM}}}]\label{thm:medvedetal}
Fix a prime $p>2$. We have
\begin{equation*}
\Delta_4(\rhobar \otimes \omega) = \begin{cases}
-\Delta_2(\rhobar) + 1 & \text{if $\rhobar = 1 \oplus \omega$;}\\
-\Delta_2(\rhobar) & \text{otherwise}.
\end{cases}
\end{equation*}
If $k\geq 4$, then 
\begin{equation*}
\Delta_{k+2}(\rhobar \otimes \omega) = -\Delta_{k}(\rhobar).
\end{equation*}
\end{thm}

The proof of this theorem uses the trace formula to establish $p$-power congruences between traces of Hecke operators to deduce isomorphisms between semi-simple (virtual) mod $p$ Hecke-modules whose dimensions encode the $\Delta_k(\rhobar)$.

\begin{cor}
Assume $\rhobar$ is modular of level $\Gamma_0(N)$ but not isomorphic to $1 \oplus \omega$. Then $\Delta_k(\rhobar) = 0$ for all $k \geq 2$ if and only if
\begin{equation*}
\rhobar|_{\Gal(\Qpbar/\Qp)} \not\simeq \begin{pmatrix} \omega & \ast \\ 0 & 1 \end{pmatrix} \otimes \unr(\pm 1) \otimes \omega^j,
\end{equation*}
for any choice of $\pm$ and any integer $j$.
\end{cor}
\begin{proof}
Since $\rhobar \not\simeq 1 \oplus \omega$, we can re-write Theorem \ref{thm:medvedetal} to see that
\begin{equation*}
\Delta_k(\rhobar) = (-1)^{\frac{k}{2}-1}\Delta_2\left(\rhobar \otimes \omega^{1-\frac{k}{2}}\right)
\end{equation*}
for all integers $k\geq 2$. The corollary then follows from Lemma \ref{lemma:sign-disc-2}.
\end{proof}

\subsection{Galois, II}\label{subsec:galois}

The final goal of this section is to discuss the periodic patterns that appear in the plot of Figure \ref{fig:5_5_Galois_representations}.  Looking at that figure (where $p=5$), one sees that along vertical lines, the plot alternates between $1 \oplus \omega$ and $\omega^2 \oplus \omega^3$ in weights $k \equiv 2 \pmod{4}$ and between $1 \oplus \omega^3$ and $\omega \oplus \omega^2$ in weights $k \equiv 0 \pmod{4}$.  What is causing this behavior?

From the perspective of Sections \ref{subsec:rbar-loci} and \ref{sec:global}, our lists of $\L$-invariants are a sampling of components of the various $\rbar$-loci.  Thus patterns in the local representations that appear as one fixes the weight and varies the valuation of the $\L$-invariant reflect how these components sit together within the entire $\P^1$ of $\L$-invariants.  

However, the sampling arising from the data in Figure \ref{fig:5_5_Galois_representations} misses many components for two reasons.  First, the data set used to create this figure only contains regular eigenforms (that is, forms whose residual representation at $p$ is reducible).  Thus, all $\rbar$-components with $\rbar$ irreducible are missed.  Second, all of the forms in this data set have {\it globally} reducible residual representations.  In particular, Theorem \ref{thm:local-global} does not apply and this could lead to further components being missed.  

To really understand how the local residual representations vary as one varies the slope of the $\L$-invariant, one needs to consider $\L$-invariants of a larger set of eigenforms.  Namely, we should consider eigenforms with globally irreducible residual representations and to have as much diversity as possible for the local action at $p$ (including both irreducible and reducible representations, split and non-split representations, etc.).

Let's focus on $p=5$ and the case of $k \equiv 2 \pmod{4}$ so that the determinant of our representations is $\omega$.  In this case, there are only three irreducible representations of $G_{\Qp}$ over $\F_5$, namely:\ $\ind(\omega_2)$, $\ind(\omega_2^9)$, and $\ind(\omega_2^{13})$. In the reducible case, there are infinitely many such representations as the image of Frobenius itself has infinitely many possibilities.  But if we instead consider the restriction of these reducible representations to the inertia group $I_p$, then the only options are $1 \oplus \omega$ and $\omega^2 \oplus \omega^3$, up to semi-simplification.

By varying the tame level $N$, we can find eigenforms representing each local possibility. Namely, we consider $\rhobar$ listed in Table \ref{table:rhobar-simple}. Their corresponding modular lifts is given by the LMFDB labels.

\begin{table}[htp]
\renewcommand{\arraystretch}{1.2}
\caption{A specification of three global Galois representations modulo 5.}
\begin{center}
\begin{tabular}{|c|c|c|c|c|}
\hline
$\rhobar$-label & $N$ & $k$ & LMFDB label & local at $5$ representation \\
\hline
$\rhobar_1$  & $14$ & 4 & 14.4.a.b &  $\begin{pmatrix} \unr(2)\omega^3 & \neq 0 \\ 0 & \unr(3) \end{pmatrix}$\\
$\rhobar_2$  & $14$ & 2 & 14.2.a.a & $\ind(\omega_2)$\\
$\rhobar_3$  & $9$ & 4 & 9.4.a.a & $\ind(\omega_2^3)$\\
\hline
\end{tabular}
\end{center}
\label{table:rhobar-simple}
\end{table}

Then, we consider five global representations:
\begin{equation}
\label{eqn:global_ordered}
\rhobar_2,\rhobar_1(1),\rhobar_3(1),\rhobar_1(3),\rhobar_2(2).
\end{equation}
The corresponding local representations up to inertia and semi-simplification are given by
\begin{equation}
\label{eqn:local_ordered}
\ind(\omega_2),1 \oplus \omega,\ind(\omega_2^9),\omega^2 \oplus \omega^3,\ind(\omega_2^{13}).
\end{equation}
We intentionally list the representations in this order, for reasons that will become clear soon.

Examining the $5$-adic $\L$-invariants in level $\Gamma_0(70)$ and $\Gamma_0(45)$, we compiled $\rhobar$-data for each $\rhobar$ in the list \eqref{eqn:global_ordered}, up to the weight $k = 44$. Since the tame level is varying, we ignore any multiplicity in the data and plot simply the valuations of the possible $\L$-invariants, in Figure \ref{fig:5_interweaving}.

\begin{figure}[htbp]
\centering
\includegraphics[scale=.75]{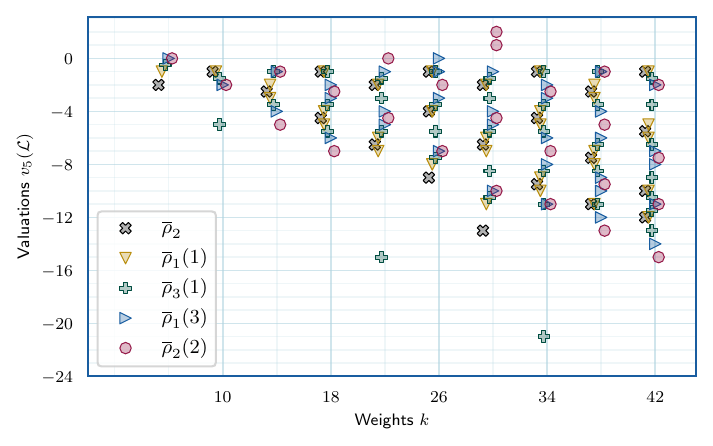}
\caption{Values of $v_p(\L_f)$ for $f$ an eigenform of weight $6\leq k \leq 44$, with $k \equiv 2 \bmod 4$, representing an eigenform lifting some $\rhobar$, up to twist, from Table \ref{table:rhobar-simple}.}
\label{fig:5_interweaving}
\end{figure}

The above plot shows incredible regularity in each fixed weight.  (Note, we offset data over a fixed weight, to make everything as easy as possible to read.) Let's focus on the particular weight $k=18$.  Starting at the most negative $v_5(\L)$ and moving upwards, the data follows the exact order in which the plot's legend is constructed. In terms of the local $\rbar$, we see the pattern
\begin{multline}
\label{eqn:k18}
\ind(\omega_2^{13}), \omega^2 \oplus \omega^3, \ind(\omega_2^9), 1 \oplus \omega, \ind(\omega_2),
1 \oplus \omega, \ind(\omega_2^9), \omega^2 \oplus \omega^3,\\ \ind(\omega_2^{13}), \omega^2 \oplus \omega^3, \ind(\omega_2^9), 1 \oplus \omega, \ind(\omega_2).
\end{multline}
First, note that the representations listed above strictly alternate between reducible and irreducible.  Second, note that $\ind(\omega_2^9)$ is always surrounded by $1 \oplus \omega$ and $\omega^2 \oplus \omega^3$ while $\ind(\omega_2)$ (resp.\ $\ind(\omega_2^{13})$) is surrounded on both sides by  $1 \oplus \omega$ (resp.\  $\omega^2 \oplus \omega^3$). Such patterns continue throughout Figure \ref{fig:5_interweaving}. Why should such regular behavior occur?

\setlength{\intextsep}{2pt}
\begin{wrapfigure}{r}{0.5\textwidth}
\centering
\includegraphics[scale=1.2]{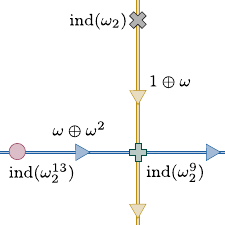}
\caption{Illustration of Emerton and Gee's stack, collapsed up to semi-simplification.}
\label{fig:EG-stack}
\end{wrapfigure}

We can start to explain this phenomenon using Figure \ref{fig:EG-stack}. What we are seeing in the data is that the $\rbar$ are arranging themselves in exactly the way they should, given the geometry of the moduli stack of 2-dimensional $\Gal(\Qpbar/\Qp)$-representations over $\Fp$, constructed by Emerton and Gee \cite{EG}.   To explain, we continue to fix $p=5$ and consider mod $p$ Galois representations with determinant equal to $\omega$ (so the modular weights are $k \equiv 2 \bmod 4$). Emerton and Gee's stack has a geometry governed by Serre weights of mod $p$ Galois representations \cite[Theorem 1.2.1]{EG}. If we collapse that geometry by considering everything up to semi-simplification we arrive at a geometric configuration with only two components.\footnote{We are not giving a moduli interpretation of this construction. This is only an illustration.} One component corresponds to mod $p$ representations that generically match $1 \oplus \omega$ (on inertia and up to semi-simplification), while the other matches $\omega^2 \oplus \omega^3$.  These two components meet at the irreducible $\ind(\omega_2^9)$.  The representation $\ind(\omega_2)$ lies on the first component while $\ind(\omega_2^{13})$ is on the second.  The Figure \ref{fig:EG-stack} illustrates the stack's geometry after the collapsing is invoked.

This geometric configuration matches the patterns in our data perfectly. The representation $\ind(\omega_2)$ is surrounded by representations that generically look like $1 \oplus \omega$, while $\ind(\omega_2^{13})$ is generically surrounded by representations that look like $\omega^2 \oplus \omega^3$. And, of course, $\ind(\omega_2^9)$ is the one sitting in between the two reducible possibilities. Returning to Theorem \ref{thm:local-global}, recall that our global data samples the $\rbar$-loci in the $\P^1$ of $\L$-invariants. In short, we are proposing (and seeing in reality) that the $\rbar$-components abut each other in a fashion perfectly mirrored by the stack. The issue of collapsing everything up to semi-simplification is explained, more or less, by Proposition \ref{prop:ribets-lemma}.

This discussion was motivated by patterns found in Figure \ref{fig:5_5_Galois_representations}. In that case, our data was data at level $\Gamma_0(5)$, where all the eigenforms lift globally {\em reducible} representations. If we ignore the irreducible $\rbar$ in the preceding discussion and remove them from \eqref{eqn:k18}, we would expect Figure \ref{fig:5_5_Galois_representations} to reveal a pattern
\begin{equation*}
\dots 1 \oplus \omega, 1 \oplus \omega, \omega^2 \oplus \omega^3, \omega^2 \oplus \omega^3, 1 \oplus \omega, 1 \oplus \omega, \omega^2 \oplus \omega^3, \omega^2 \oplus \omega^3, \dots
\end{equation*}
in the general weight fiber of the scatter plot. If you look closely, you will see this does not happen. In fact, Figure \ref{fig:5_5_Galois_representations} shows a {\em strict} alternation between $1 \oplus \omega$ and $\omega^2 \oplus \omega^3$.  We resolve this issue as follows. What we are seeing is a concrete manifestation of the failure of Theorem \ref{thm:local-global} to hold in the context of globally reducible Galois representations. What is fascinating to us is that this failure is {\em completely systematic} --- according to our global data we might suspect that the $\Gamma_0(5)$-eigenforms hit roughly {\em half} the components on the $\rbar$-loci and the half that they miss are systematically arranged.

\begin{remark}
While writing the above discussion, we noticed a simple heuristic for why Theorem \ref{thm:local-global} fails for globally reducible Galois representations.  To explain, let us start by observing that
\begin{equation*}
\dim_{\C} S_k(\Gamma_0(5),\C)^{\new} \approx \frac{k}{3}.
\end{equation*}
Here we write $(-)^{\new}$ to mean the span of newforms. Switching to $p=5$ and the $p$-adics, we use $(-)_{\rhobar}$ to denote the $\rhobar$-component. In these notations, if $\rhobar = \omega^j \oplus \omega^{j+1}$, and $\rhobar$ appears in weight $S_k(\Gamma_0(5))$, then we have
\begin{equation*}
\dim_{\Qpbar} S_k(\Gamma_0(5),\Qpbar)^{\new}_{\rhobar} \approx \frac{k}{6}.
\end{equation*}

Now suppose instead that $\rhobar$ is modular of level $\Gamma_0(N)$ with $N$ prime to $5$, and $\rhobar$ satisfies the assumptions of Theorem \ref{thm:local-global} (and so, in particular, is globally irreducible), but still $\rhobar|_{\Gal(\Qpbar/\Qp)} = \omega^j \oplus \omega^{j+1}$. Then,
\begin{equation*}
\dim_{\Qpbar} S_k(\Gamma_0(5N),\Qpbar)_{\rhobar} \approx \frac{\mu k}{6},
\end{equation*}
where $\mu = \dim_{\Qpbar} S_2(\Gamma_0(N),\Qpbar)_{\rhobar}  + \dim_{\Qpbar} S_{p+1}(\Gamma_0(N),\Qpbar)_{\rhobar}$. For this dimension count, see \cite[Corollaries 6.11 and 6.17]{BP-Ghost2}. In any case, Serre's conjecture implies $\mu \geq 1+1=2$ and so  $\dim_{\Qpbar} S_k(\Gamma_0(5N),\Qpbar)_{\rhobar}$ is asymptotically at least $k/3$. 

To summarize, given the same local Galois-theoretic setup and assuming minimal possible global multiplicities, the number of eigenforms lifting a globally irreducible $\rhobar$ will be roughly twice the number you would find lifting a globally reducible $\rhobar$. In terms of deformation spaces, one option is that only ``half'' the local components are hit by modular points in the globally reducible situations. The other option would be the modular points in the globally irreducible case overrepresent the local components. The second option is ruled out by explicit computations in low weights, lending credence to the first option. Of course, this kind of numerology does nothing to explain why the missing ``half'' arranges itself so systematically with respect to weight variation, in the case where the first option is valid.
\end{remark}

\begin{remark}
There is no reason to limit the preceding discussion to $\L$-invariants and semi-stable Galois representations. In fact, we could have also carried out the same analysis of slope data for $v_p(a_p(f))$ with $f$ an eigenform of level prime to $p$. We in fact did that, as we worked through the discussion here and our conclusions are the same.

We are also not the only ones to notice this kind of cyclic nature. For instance, Ghate's zig zag conjecture \cite{Ghate-ZigZag} seems to be a highly-tuned and precise version of the systematic behavior we are explaining here. As far as we know, Ghate has not expressed his conjecture in terms of the geometry of the moduli stack of Galois representations, but the pattern he proposes (at a qualitative level) {\em is} explained by the geometry.
\end{remark}

\section{The distribution of $\L$-invariants}\label{sec:conjecture}

The primary goal of this section is to formulate a conjecture on the distribution of $v_p(\L_f)$ as $f$ ranges over all $\Gamma_0$-eigenforms of weight $k \rightarrow +\infty$. After formulating the $\L$-invariant conjecture, we provide numerical {\em and} heuristic evidence, which links the distribution of slopes of $\L$-invariants to Gouv\^ea's distribution conjecture.

\subsection{Gouv\^ea's distribution conjecture}

To frame the discussion, we first formally recall the $\rhobar$-version of Gouv\^ea's conjecture on slopes of $a_p$ from Section \ref{subsec:nonarch-dist}. Recall $\Gamma_0 = \Gamma_0(Np)$ with $p \nmid N$. Assume
\begin{equation*}
\rhobar : \Gal(\Qbar/\Q) \rightarrow \GL_2(\Fpbar)
\end{equation*}
is modular of level $\Gamma_0(N)$. Then for each number $T$ we can look at
\begin{equation*}
\mathbf x_T(\rhobar) = \left\{\frac{p+1}{k} \cdot v_p(a_p(f)) ~\Big|~ \parbox{7cm}{\centering $f$ is a $\Gamma_0(N)$-newform of weight  $k \leq T$ such that $\rhobar_f \simeq \rhobar$} \right\}.
\end{equation*}

\begin{conj}[Gouv\^ea's distribution conjecture]\label{conj:gouveas}
The sets $\mathbf x_T(\rhobar)$ are equidistributed on $[0,1]$ for Lebesgue measure as $T\rightarrow \infty$.
\end{conj}

The history of Conjecture \ref{conj:gouveas} (and in particular, precision as to exactly what Gouv\^ea considered) is contained in Section \ref{subsec:nonarch-dist}, as is an explanation that Conjecture \ref{conj:gouveas} was recently proven by Liu, Truong, Xiao, and Zhao whenever $\rhobar$ is reducible on a decomposition $p$ and has sufficiently generic Serre weights. See Theorem \ref{thm:LTXZ-application}.

We add, given the extra context provided by Section \ref{sec:reanalysis}, that the Liu--Truong--Xiao--Zhao paper also proves a rather strict law.  Under their assumptions, there is an {\it a priori} containment $\mathbf x_T(\rhobar) \subseteq [0,1]$. See \cite[Theorem 1.12]{LTXZ-Proof}.

\subsection{Distribution of $\L$-invariants.}
\label{sec:disL}

In this section, we state our most specific conjecture on the distribution of $\mathcal L$-invariants. Once again, $\Gamma_0 = \Gamma_0(Np)$ with $N$ not divisible by $p$. Let $S_k(\Gamma_0)^{\pnew}$ be the subspace of $S_k(\Gamma_0)$ spanned by eigenforms that are new at the prime $p$ and $S_k(\Gamma_0)^{\pm}$ the subspace of $S_k(\Gamma_0)^{\pnew}$ on which $a_p = \pm p^{\frac{k}{2}-1}$. We then consider a mod $p$ Galois representation 
\begin{equation*}
\rhobar: \Gal(\Qbar/\Q) \rightarrow \GL_2(\Fpbar).
\end{equation*}
We assume that it is modular of level $\Gamma_0(N)$.

\begin{conj}
\label{conj}
For each sign $\pm$, define
\begin{equation*}
\mathbf y_T^{\pm}(\rhobar) = \left\{\frac{2(p+1)}{k(p-1)} \cdot v_p(\L_f) ~\Big|~ \parbox{7cm}{\centering $f \in S_k(\Gamma_0)^{\pm}$ is an eigenform of weight  $k \leq T$ such that $\rhobar_f \simeq \rhobar$} \right\}.
\end{equation*}
Then, the sets $\mathbf y_T^{\pm}$ are equidistributed on $[-1,0]$ for Lebesgue measure as $T \rightarrow \infty$.
\end{conj}

We first compare and contrast our Conjecture \ref{conj} with Conjecture \ref{conj:gouveas}. Obviously, the data gathered are different numerical quantities over different ranges of eigenforms. Then, there are two more numerical differences.

 First, in Conjecture \ref{conj} we predict an end result supported on $[-1,0]$, while Conjecture \ref{conj:gouveas} predicts an end result supported on $[0,1]$. If we replaced $\L_f$ by $\L_f^{-1}$ in our conjecture, this difference would go away. This replacement may have conceptual advantages, in terms of Section \ref{subsec:rbar-loci}, as $\L^{-1} \approx 0$ is where the semi-stable, non-crystalline, loci in local deformation spaces intersect crystalline loci. But, leaving $\L_f$ the way it is also has the practical advantage that it reminds us that $v_p(\L_f) \leq 0$ almost all the time.

Second, Conjectures \ref{conj:gouveas} and \ref{conj} use different normalizing constants.  Namely, the normalizing constants for the two conjectures are respectively
$$
C_{a_p} = \frac{p+1}{k} \quad \text{and} \quad C_{\L} = \frac{2(p+1)}{k(p-1)}.
$$
These normalizing constants are related by the equation
\begin{equation}\label{eqn:basic-scale}
C_{\L} = \frac{2}{p-1} C_{a_p}.
\end{equation}
This relationship is consistent with the amount of data being gathered. To explain this, we write $S_{\rhobar}$ for the $\rhobar$-isotypic component in a space of cuspforms. Thus, in the $a_p$-case, we range over eigenforms lying $S_k(\Gamma_0(N))_{\rhobar}$, while in the $\L$-invariant case, we range over eigenforms in $S_k(\Gamma_0)^{\pm}_{\rhobar}$. Ignoring the fixed $\rhobar$ and ignoring the sign $\pm$, there is a classical dimension formula that says
\begin{equation}
\label{eqn:dimform}
\dim S_k(\Gamma_0)^{\pnew}  = (p-1)\dim S_k(\Gamma_0(N)) + O(1).
\end{equation}
And, indeed, \eqref{eqn:dimform} continues to hold when one fixes the $\rhobar$-isotypic components, as proven in \cite[Corollary 6.11]{BP-Ghost2}. (The equation relating $Q_{d_t^{\new}}$ to $Q_{d_t}$ is the relevant one in {\em loc.\ cit.}) Even more, forthcoming work of Anni, Ghitza, and Medvedovsky (\cite{AGM}) proves that 
\begin{equation*}
\dim S_k(\Gamma_0)^{\pm}_{\rhobar} = \frac{1}{2}\dim S_k(\Gamma_0)^{\pnew}_{\rhobar} + O(1),
\end{equation*}
as well. Therefore, in simply counting the number of eigenforms for which we gather data, we find that
\begin{equation*}
\dim S_k(\Gamma_0)^{\pm}_{\rhobar} \approx \frac{p-1}{2} \dim S_k(\Gamma_0(N))_{\rhobar},
\end{equation*}
which is the inverse relationship to  \eqref{eqn:basic-scale}. This logic does not explain the {\em a priori} choice of either $C_{a_p}$ or $C_{\L}$, only their ratio.

\subsection{Scatter plots for distributional conjecture}\label{subsec:plot-dist}

To illustrate our evidence toward Conjecture \ref{conj}, we offer normalized scatter plots in level $\Gamma_0(3)$, $\Gamma_0(5)$, $\Gamma_0(7)$, and $\Gamma_0(11)$.

\begin{figure}[htbp]
\centering
\includegraphics[scale=.75]{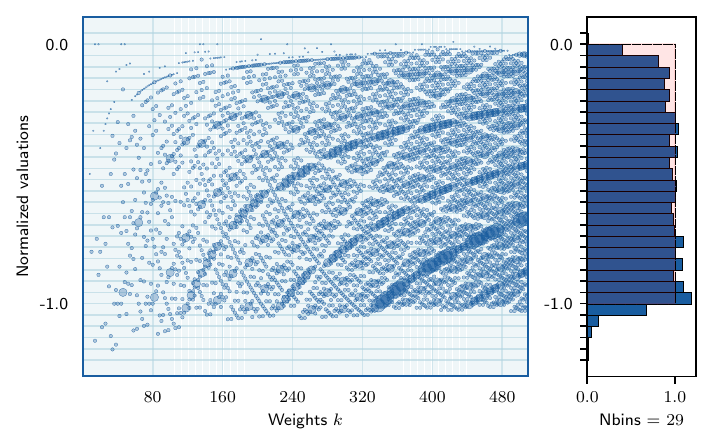}
\caption{Normalized $3$-adic slopes of $\L$-invariants in level $\Gamma_0(3)$ for eigenforms of weight $6 \leq  k \leq 508$. (Total data:\ 10752 eigenforms.)}
\label{fig:3_3_distribution}
\end{figure}

\begin{figure}[htbp]
\centering
\includegraphics[scale=.75]{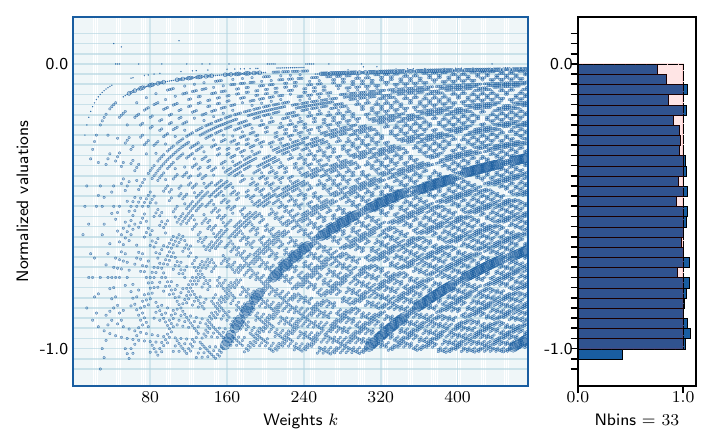}
\caption{Normalized $5$-adic slopes of $\L$-invariants in level $\Gamma_0(5)$ for eigenforms of weight $10 \leq k \leq 472$. (Total data:\ 18560 eigenforms.)}
\label{fig:5_5_distribution}
\end{figure}

\begin{figure}[htbp]
\centering
\includegraphics[scale=.75]{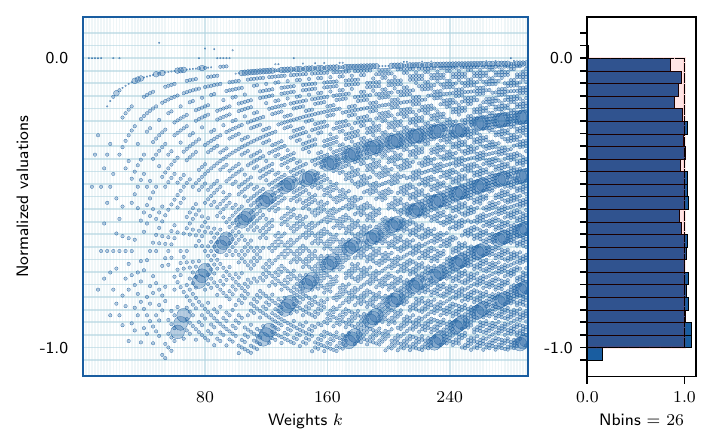}
\caption{Normalized $7$-adic slopes of $\L$-invariants in level $\Gamma_0(7)$ for eigenforms of weight $4 \leq k \leq 290$. (Total data:\ 10511 eigenforms.)}
\label{fig:7_7_distribution}
\end{figure}

\begin{figure}[htbp]
\centering
\includegraphics[scale=.75]{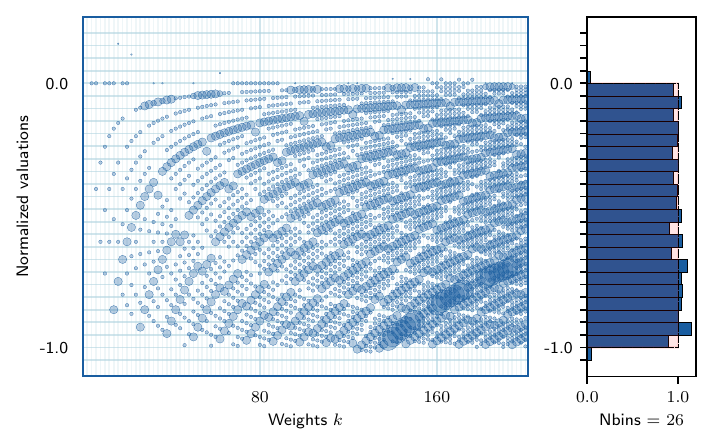}
\caption{Normalized $11$-adic slopes of $\L$-invariants in level $\Gamma_0(11)$ for eigenforms of weight $4 \leq k \leq 200$. (Total data:\ 8322 eigenforms.)}
\label{fig:11_11_distribution}
\end{figure}

\subsection{Heuristic for $\L$-invariants as gradients.}\label{subsec:heuristic}

Conjecture \ref{conj} is supported by data. However, it is also supported by an {\em a priori} prediction. In this section, we explain a heuristic that we used to the predict Conjecture \ref{conj} before extensive data was gathered. There is nothing in this section that should be labeled a theorem, but we hope that this heuristic gives some theoretical underpinnings behind the formulation of Conjecture \ref{conj} and its connection to Gouv\^ea's conjecture Conjecture \ref{conj:gouveas}. The heuristic is based on modeling $\L$-invariants as gradients of the $a_p$-function along the eigencurve.   The logic behind the heuristic can be traced to emails, in 2007, between the second author and Kevin Buzzard. (These emails are also where the second author agreed to smash together computer code to try to calculate some $\L$-invariants.) The logic also motivates arguments on bounding constant slope radii of Coleman families in \cite{Bergdall-bounds}. 

Let $f$ denote a $p$-new eigenform of level $\Gamma_0$ with weight $k_0$ and $a_p(f) = \pm p^{\frac{k_0}{2}-1}$.  Then, $f = \mathbf{f}|_{k=k_0}$ where
\begin{equation*}
\mathbf f = q + a_2(k)q^2 + a_3(k)q^3 + \dotsb
\end{equation*}
is the $q$-expansion of the family of eigenforms arising from a small piece of Coleman and Mazur's eigencurve \cite{ColemanMazur} (see for instance \cite[Corollary B5.7.1]{ColemanBanach}). The function $a_p(k)$ has an analytic expansion
\begin{align*}
a_p(k) &= a_p(k_0) + a_p'(k_0)(k-k_0) + O((k-k_0)^2) \\
 &= \pm p^{\frac{k_0}{2}-1} + a_p'(k_0)(k-k_0) + O((k-k_0)^2),
\end{align*}
and, as mentioned in the introduction as (\ref{eqn:ap}), we have the famous formula
\begin{align}\label{eqn:L-invariant-explicit}
\mathcal L_f &= -2 \frac{a_p'(k_0)}{a_p(k_0)}\\
&= \mp 2 p^{1-\frac{k_0}{2}}a_p'(k_0).
\end{align}
In particular,
\begin{align}
\label{eqn:ap_approx}
a_p(k) &= \pm p^{\frac{k_0}{2}-1}\left( 1 -  \frac{\L_f}{2} \cdot (k-k_0)\right) + O((k-k_0)^2).
\end{align}
From this formula we see that if $k$ is very close to $k_0$, then $v_p(a_p(k))$ will simply equal $v_p(a_p(k_0)) = \frac{k_0}{2}-1$.  However, if $k$ is close but not too close to $k_0$, one could hope that how $v_p(a_p(k))$ varies with $k$ might shed light on $\L_f$.

\begin{remark}
A careful reader will note a slight of hand in the setup. The family $\mathbf f$ is a family of eigenforms of level $\Gamma_0$. At $k_1 \neq k_0$ a classical integer, $f_{k_1} := \mathbf{f}_{k=k_1}$ is an oldform arising from the {\em $p$-refinement} of a level $\Gamma_0(N)$-eigenform $g$. What we are calling $a_p(k_1)$ is actually one of the two roots of the $p$-th Hecke polynomial $x^2 - a_p(g)x + p^{k_1-1}$. Note that Gouv\^ea's distribution conjecture predicts with 100\% certainty that $v_p(a_p(g)) = v_p(\alpha_p)$ where $\alpha_p$ is the root of the smaller valuation. Since the sum of the valuations of the roots is $k_1-1$, once $k_1 \gg k_2$, we would have with near certainty that $v_p(a_p(g)) = v_p(a_p(k_1))$.
\end{remark}

We illustrate the above idea with a numerical example.  Take $p=3$ and $k=52$ in tame level $N = 1$.  Table \ref{table:slopes-near-52} lists the slopes of eigenforms of level $\Gamma_0$ which appear in a sequence of weights $3$-adically approaching 52. The slopes in weight $52$ appear at the bottom of the table. We visually cut the table at weight $52 + 2 \cdot 3^{12}$, as the listed data after that point is constant and matches the data in weight $52$.\footnote{Actually the data presented here is computed via the ghost series as the weights which appear in this table are too large to effectively compute true slopes.  Thus, this data is contingent on the ghost conjecture which is not yet known in this case.}

{\small 
\begin{table}[htp]
\renewcommand{\arraystretch}{1.2}
\caption{List of slopes $v_3(a_3(f))$ (in increasing order) for eigenforms $f \in S_k(\Gamma_0(3))$. The fifth through twelfth slopes are bolded, for illustration purposes.}
\begin{center}
\begin{tabular}{|c|c|}
\hline
$k$ & Slopes\\
\hline
$52+2 \cdot 3^{ 3 }$& 3, 6, 8, 12, \textbf{16, 18, 21, 24, 52, 52, 52, 52}, 52, 52, 52, 52, $\dots$\\
$52+2 \cdot 3^{ 4 }$& 3, 6, 8, 12, \textbf{17, 19, 22, 25, 25, 28, 31, 33}, 39, 43, 45, 48, $\dots$ \\
$52+2 \cdot 3^{ 5 }$& 3, 6, 8, 12, \textbf{18, 20, 23, 25, 25, 27, 30, 32}, 39, 43, 45, 48, $\dots$ \\
$52+2 \cdot 3^{ 6 }$& 3, 6, 8, 12, \textbf{19, 21, 24, 25, 25, 26, 29, 31}, 39, 43, 45, 48, $\dots$ \\
$52+2 \cdot 3^{ 7 }$& 3, 6, 8, 12, \textbf{20, 22, 25, 25, 25, 25, 28, 30}, 39, 43, 45, 48, $\dots$ \\
$52+2 \cdot 3^{ 8 }$& 3, 6, 8, 12, \textbf{21, 23, 25, 25, 25, 25, 27, 29}, 39, 43, 45, 48, $\dots$ \\
$52+2 \cdot 3^{ 9 }$& 3, 6, 8, 12, \textbf{22, 24, 25, 25, 25, 25, 26, 28}, 39, 43, 45, 48, $\dots$ \\
$52+2 \cdot 3^{ 10 }$& 3, 6, 8, 12, \textbf{23, 25, 25, 25, 25, 25, 25, 27}, 39, 43, 45, 48, $\dots$ \\
$52+2 \cdot 3^{ 11 }$& 3, 6, 8, 12, \textbf{24, 25, 25, 25, 25, 25, 25, 26}, 39, 43, 45, 48, $\dots$ \\
\hline\hline
$52+2 \cdot 3^{ 12 }$& 3, 6, 8, 12, \textbf{25, 25, 25, 25, 25, 25, 25, 25}, 39, 43, 45, 48, $\dots$ \\
$52+2 \cdot 3^{ 13 }$& 3, 6, 8, 12, \textbf{25, 25, 25, 25, 25, 25, 25, 25}, 39, 43, 45, 48, $\dots$ \\
$\vdots$ & $\vdots$\\
\hline\hline
$52$ & 3, 6, 8, 12, \textbf{25, 25, 25, 25, 25, 25, 25, 25}, 39, 43, 45, 48\phantom{, $\dots$} \\
\hline
\end{tabular}
\end{center}
\label{table:slopes-near-52}
\end{table}
}

In weight $k_0=52$ we are only interested in the slope $25 = \frac{k_0}{2}-1$ of the eight newforms. These are the fifth through the twelfth slopes, when ordered as in the table. Let's focus our attention on the fifth slope in each weight in the table, which is the first bolded slope.  As we move from the bottom of the table upwards,  this fifth slope is 25 in weights $52 + 2 \cdot 3^j$ for $j \geq 12$ and for smaller $j$ the slope decreases by 1 in each line in the table.   The same phenomenon occurs for the sixth through eighth slopes but now the cutoff points after which the slopes begin to decrease (by one) is $j=10$, $7$ and $4$, respectively. This behavior suggests the analytic expansions in \eqref{eqn:ap_approx} of $a_p(k)$ corresponding to these newforms in weight 52 behaves like a linear function.  Indeed, if $a_p(k)$ were linear, by \eqref{eqn:ap_approx}, we would have 
\begin{align*}
a_p(k) &= \pm p^{\frac{k_0}{2}-1}\left( 1 -  \frac{\L_f}{2} \cdot (k-k_0)\right) 
\end{align*}
and, in particular, as $k$ moves further away from $k_0$ the valuation of $a_p(k)$ would steadily decrease.  Precisely, the linear behavior would result in
\begin{eqnarray}
\label{eqn:changeinL}
v_p(a_p(k)) = \begin{cases} \displaystyle \frac{k_0-2}{2} & 
\text{if~} v_p(k-k_0) > - v_p(\L_f) \\
\displaystyle  \frac{k_0-2}{2} - t & 
\text{if~}  v_p(k-k_0) = -v_p(\L_f) - t \quad (t>0).
\end{cases}
\end{eqnarray}
Note  that under \eqref{eqn:changeinL}, the valuations of $\L$-invariants determine the cutoff points where the slopes begin to decrease in Table \ref{table:slopes-near-52}.  This exactly matches the numerical data we collected on $\L$-invariants! Specifically, in our example of weight 52 and level $\Gamma_0(3)$, the valuations of the $\L$-invariants are $-4, -4, -7, -7, -10, -10, -12, -12$ which are exactly the cutoffs (doubled) we observed in the table above.

This heuristic, though, must be taken with a large grain of salt for at least two reasons. Firstly, if we look at the second half of the slopes in weight 52 (the ninth through twelfth slopes),  an opposite phenomenon occurs. The slopes actually {\em increase} by one as $k$ moves further from $k_0=52$.  Secondly, there is absolutely no reason to believe that the power series $a_p(k)$ converges for $k$ so far away from $k_0$. In fact, since $a_p(-)$ is a non-vanishing function on the eigencurve, the norm $|a_p(k)|$ {\em must} be constant on any disc over which the expansion \eqref{eqn:ap_approx} converges. (This is a consequence of the Weierstrass preparation theorem.) Nonetheless, the behavior in this example is mirrored in countless other examples.

We formalize this discussion in the following heuristic.  Fix a global mod $p$ Galois representation $\rhobar$ occurring in level $\Gamma_0$. We write the following for various dimensions. 
\begin{equation*}
d_k =\dim S_k(\Gamma_0(N))_{\rhobar}, \;\;\;\; d_k^{\new} = \dim S_k(\Gamma_0)^{p\text{-new}}_{\rhobar}, \;\;\;\; d_{k,p} = \dim S_k(\Gamma_0)_{\rhobar}.
\end{equation*}
 Let $S_k^{\ve}(\Gamma_0)_{\rhobar}^{\new}$ denote the subspace of $S_k(\Gamma_0)_{\rhobar}^{\new}$ generated by eigenforms whose sign of $a_p$ equals $\varepsilon = \pm$.

\begin{heuristic}
\label{heuristic:slopesvsL}
Fix $k_0 \geq 2$ and choose $k \equiv k_0 \pmod{p-1}$ large enough so that $d_k > d_{k_0,p}$.  Denote the slopes which occur in $S_k(\Gamma_0(N))_{\rhobar}$ by $s_1 \leq s_2 \leq \dots \leq s_{d_k}$.  
If $v_p(k-k_0) > \log(k_0)/\log(p)$, 
then, for $\varepsilon = \pm$, there is some labeling $f_1, f_2, \dots$ of the newforms in $S^{\ve}_{k_0}(\Gamma_0)^{\new}_{\rhobar}$ such that 
\begin{equation}
\label{eq:slopevsL}
s_{i+d_{k_0}} = 
\begin{cases}
\ds v_p(\L_{f_i}) + \frac{k_0-2}{2} + v_p(k-k_0) & \text{if~} v_p(k-k_0) < -v_p(\L_{f_i}) \\
\ds \frac{k_0-2}{2}  & \text{if~} v_p(k-k_0) \geq -v_p(\L_{f_i}),\\
\end{cases}
\end{equation}
for  $1 \leq  i \leq \dim S^{\ve}_{k_0}(\Gamma_0)^{\new}_{\rhobar}$.
\end{heuristic}

Some remarks are in order.  First, the precise bound $v_p(k-k_0) > \log(k_0)/\log(p)$ might appear surprising, but it (a) matches numerical computations, and (b) is consistent with the following observation. If the above heuristic holds, then the slopes appearing in \eqref{eq:slopevsL} do not depend directly on $k$, but rather on $v_p(k-k_0)$.  From the perspective of the ghost series, such independence of slopes would occur if $k$ is closer to $k_0$ than to any of the zeroes of the coefficients of the ghost series in the indices up to $d_{k_0,p}$.  This condition is guaranteed by the inequality $v_p(k-k_0) > \log(k_0)/\log(p)$.

Second, for $k$  such that 
\begin{equation}
\label{eqn:kbound}
\log(k_0)/\log(p) < v_p(k-k_0) < -v_p(\L_f),
\end{equation} 
solving \eqref{eq:slopevsL} for $v_p(\L_f)$ gives
\begin{equation}
\label{eqn:valL}
v_p(\L_f) = s_{i+d_{k_0}} - \frac{k_0-2}{2} - v_p(k-k_0).
\end{equation}
This equation is the key for us as it relates the valuation of $\L_f$ to slopes, and slopes can be more readily controlled.  However, it is possible that no such $k$ exists 
satisfying \eqref{eqn:kbound}.   Indeed, if say $v_p(\L_f) = 0$ then \eqref{eqn:kbound} clearly cannot be satisfied.  Fortunately, our numerical data suggests that $v_p(\L_f)$ is typically quite negative and, in weight $k_0$, is very rarely larger than $-\log(k_0)/\log(p)$.  In what follows, we will assume that there exists a weight $k$ satisfying \eqref{eqn:kbound} and hope that this excludes only a thin set of $\L$-invariants.

We now state a second heuristic on the sizes of slopes.

\begin{heuristic}\label{heuristic:slope-size}
If $s_1\leq s_2 \leq \dotsb \leq s_{d_k}$ denote the slopes in $S_k(\Gamma_0(N))_{\rhobar}$, then
\begin{equation*}
s_i = \frac{k}{p+1}\cdot \frac{i}{d_k} + O(\log k).
\end{equation*}
\end{heuristic}

Where does this heuristic come from?  Recall that  Gouv\^ea's distribution conjecture (Conjecture \ref{conj:gouveas}) predicts that the normalized slopes $\frac{p+1}{k} \cdot s_i$ are uniformly distributed between $0$ and $1$ as $k \to \infty$.  If we remove the $O(\log k)$-term in the above heuristic, it would be asserting that the $i$-th normalized slope equals $\frac{i}{d_k}$.  In particular,  these normalized slopes would be perfectly, evenly, placed between $0$ and $1$ for every $k$.  Such a claim is certainly too strong and thus we need some error term for which we choose $O(\log k)$.  Heuristic \ref{heuristic:slope-size} is thus a strengthening of Gouvea's distribution conjecture.  Moreover, this heuristic follows from the author's ghost conjecture when $\rhobar$ is regular (see  \cite[Theorem 3.1(a)]{BP-Ghost2}) with this $O(\log k)$ error term.

To connect back to $\L$-invariants, take a minimal $k$ satisfying \eqref{eqn:kbound} and note that $k \approx k_0 + (p-1)p^{\log(k_0)/\log(p)} = pk_0$ so that $O(k) = O(k_0)$.
Combining Heuristic \ref{heuristic:slope-size} with \eqref{eqn:valL} then yields
\begin{equation}
\label{eqn:Lformula}
v_p(\L_f) = \frac{k}{p+1} \cdot \frac{i+d_{k_0}}{d_k} - \frac{k_0}{2} + O(\log k_0).
\end{equation}
for $i=1$ to $i = \dim S^{\ve}(\Gamma_0)^{\new}_{\rhobar} \approx d_{k_0}^{\new}/2$.  Here we pulled $v_p(k-k_0)$ and constants into the $O(\log k_0)$-term.

Note that $v_p(\L)$ varies linearly with $i$ in \eqref{eqn:Lformula}.  Thus to understand the distribution of $v_p(\L)$, we can simply evaluate 
\eqref{eqn:Lformula} at the smallest and largest values of $i$.

To this end, for $i=1$, we have
\begin{align*}
v_p(\L) &\approx \frac{k}{p+1} \cdot \frac{d_{k_0}}{d_k} - \frac{k_0}{2}
\approx \frac{k}{p+1} \cdot \frac{k_0}{k} - \frac{k_0}{2}\\
&= -\frac{(p-1)k_0}{2(p+1)}.
\end{align*}
Here we used the fact that $d_k$ grows linearly with $k$ up to bounded error (\cite[Proposition 6.9(a)]{BP-Ghost2}).   Note that our normalizing constant 
$$
C_{\L} = \ds\frac{2(p+1)}{p-1}
$$ 
has appeared!

For $i \approx d_{k_0}^{\new}/2$, we have 
\begin{align*}
v_p(\L) 
&\approx \frac{k}{p+1} \cdot \frac{d_{k_0}+\frac{d_{k_0}^{\new}}{2}}{d_k} - \frac{k_0}{2}\\
&\approx \frac{k}{p+1} \cdot \frac{k_0+\frac{p-1}{2} \cdot k_0}{k} - \frac{k_0}{2} \\
&=k_0 \left( \frac{1+ \frac{p-1}{2}}{p+1} - \frac{1}{2} \right) = 0.
\end{align*}
Here we have used  \eqref{eqn:dimform}, that $d_{k_0}^{\new} \approx (p-1) d_{k_0}$.

Heuristics \ref{heuristic:slopesvsL} and \ref{heuristic:slope-size} thus imply that 
$$
\frac{2(p+1)}{k_0(p-1)} v_p(\L) \text{ are uniformly distributed over } [-1,0]
$$
as $k_0 \to \infty$ which is exactly the statement of Conjecture \ref{conj}.

We also offer further numerical data for the reader who is skeptical of Heuristic \ref{heuristic:slopesvsL}.   We repeated the numerical experiment explained earlier for $p=3$ and $k=52$, but now for $p = 5$ and weights $2 \leq k_0 \leq 300$. By looking at the change of the multiplicity of the slope $\frac{k}{2}-1$ in weights $k_0 + 4 \cdot 5^j$ as $j$ varies, we used 
Heuristic \ref{heuristic:slopesvsL} to predict the valuations of $\L$-invariants in weight $k_0$.  We note again that the size of the weights which we used were astronomically large --- for $k_0 \approx 240$, we needed $j \approx 80$ so that we were in weight around $5^{80}$.  It is infeasible to directly compute in weights this large, but instead we used the ghost series which has no problem handling such large weights.

\begin{figure}[htbp]
\centering
\includegraphics[scale=.6]{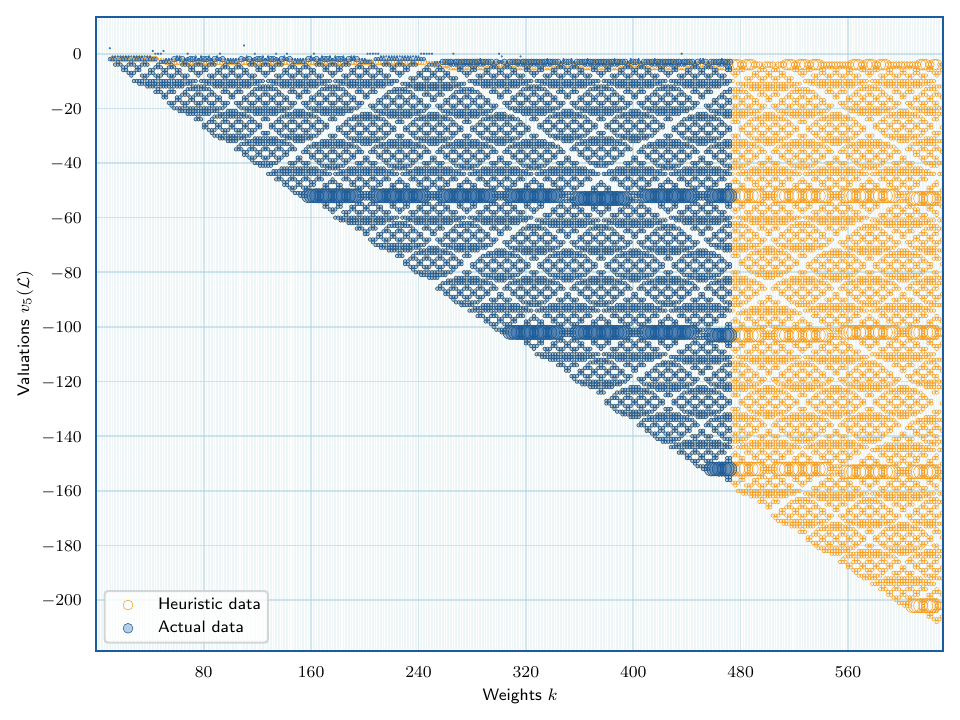}
\caption{$5$-adic slopes of $\L$-invariants in level $\Gamma_0(5)$ for $10 \leq k \leq 472$ (in blue) versus the heuristic predictions for $10 \leq k \leq 630$. Total real data:\ 18560 eigenforms. Total heuristic data: 33069 ``eigenforms''.}
\label{fig:5_5_heuristic}
\end{figure}

The result of our experiment is plotted in Figure \ref{fig:5_5_heuristic}. There, the darker data represents actual gathered data on $\L$-invariants. Underneath that data, we've plotted what the heuristic predicts. Except for $v_p(\L) \approx 0$, you can see a {\em perfect} matching between the heuristic calculations and the actual calculations. Returning to the discussion above, the failure to predict $v_p(\L) \approx 0$ makes sense --- that would involve predicting the absolutely largest $p$-adic families one could imagine.

\begin{remark}
In \cite{An}, assuming the ghost conjecture for $\rhobar$, Jiawei An proves that Heuristic \ref{heuristic:slopesvsL} holds for half of the newforms (with $O(\log k)$ possible exceptions).  In fact, he also proves that for the other half of the newforms, the slopes increase as one moves away from $k_0$ (as we observed in Table \ref{table:slopes-near-52}).  However, he has not yet established that this splitting matches the splitting given by the sign of $a_p$.
\end{remark}

\section{Further questions}\label{sec:other}

We end with just a few questions.

\subsection{High multiplicity slopes}\label{subsec:caterpillars}

In Figure \ref{fig:3_3_raw}, one sees that in level $\Gamma_0(3)$, some choices of $v = v_p(\L)$ occur with particularly high frequency. Namely, you can see that around $v\approx -30$, $v  \approx -60$, $v \approx -85$ there are particular bands of data that occur point-by-point with high multiplicity. In the $\Gamma_0(5)$-data in Figure \ref{fig:5_5_raw} these occur for $v \approx -10, -20, -30$. For $\Gamma_0(11)$, the first extremely obvious band occurs for $v \approx -55$. What are these numbers? Why are they occurring?

To answer the first question, we performed the following experiment. For each $v$, we gathered the total multiplicity of $v_p(\L)$ occurring in the closed intervals $[v-2,v+2]$. While $2$ is somewhat arbitrary, recording multiplicities over intervals of positive length captures clustering we see in our figures, for instance around around $v_p(\L) = -10, -11, -12$ in Figure \ref{fig:5_5_raw}. The results for levels $\Gamma_0(3)$, $\Gamma_0(5)$, and $\Gamma_0(7)$ are displayed, respectively, in Figures \ref{fig:3_3_caterpillar}, \ref{fig:5_5_caterpillar}, and \ref{fig:7_7_caterpillar}. 

In those figures, the top plots represent the total multiplicities calculated for each $v = -v_p(\L)$. The bottom plots represent multiplicities {\em normalized} in the following way. As suggested in Section \ref{subsec:thresholds} and formalized in Conjecture \ref{conj}, we expect that over a range $0 \leq k \leq k_{\max}$, the multiplicity of the slope $v=-v_p(\L)$ is proportional to $k_{\max} - v/C_{\L}$ where $C_{\L} = \frac{p-1}{2(p+1)}$. The raw data collected determines the constant of proportionality, and therefore we can take the ratio of the actual slope multiplicities to their expected valued.  That ratio is plotted as the bottom half of each figure.

\begin{figure}[htbp]
\centering
\includegraphics[scale=0.75]{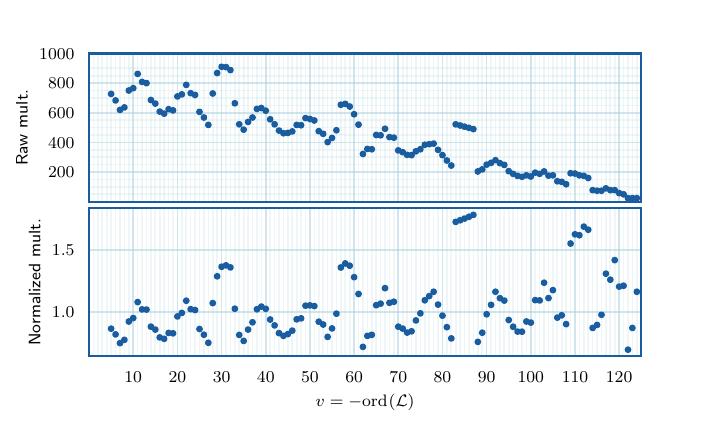}
\caption{Multiplicities (vertical) versus $v_p(\L)=-v$ (horizontal) in level $\Gamma_0(3)$. Top plot is raw multiplicities. Bottom plot is normalized multiplicities.}
\label{fig:3_3_caterpillar}
\end{figure}

\begin{figure}[htbp]
\centering
\includegraphics[scale=0.75]{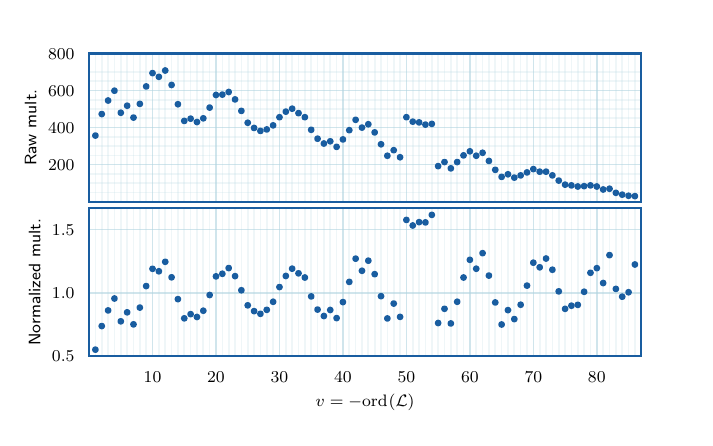}
\caption{Multiplicities (vertical) versus $v_p(\L)=-v$ (horizontal) in level $\Gamma_0(5)$. Top plot is raw multiplicities. Bottom plot is normalized multiplicities.}
\label{fig:5_5_caterpillar}
\end{figure}

\begin{figure}[htbp]
\centering
\includegraphics[scale=0.75]{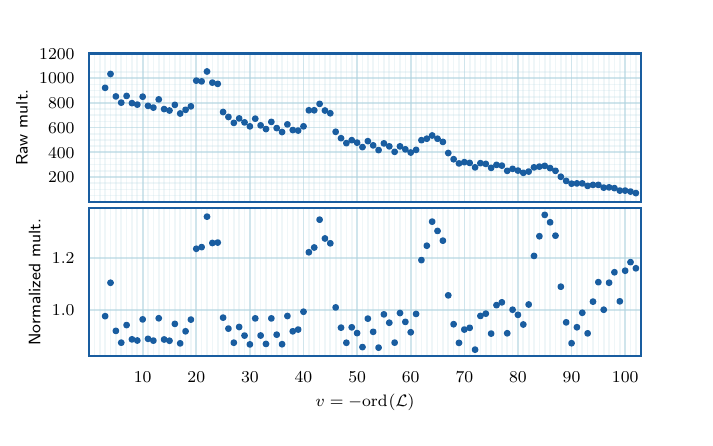}
\caption{Multiplicities (vertical) versus $v_p(\L)=-v$ (horizontal) in level $\Gamma_0(7)$. Top plot is raw multiplicities. Bottom plot is normalized multiplicities.}
\label{fig:7_7_caterpillar}
\end{figure}

As predicted, our data set reveals systematic biases towards certain valuations for $\L$-invariants. The pattern for which valuations are overrepresented is most clearly seen starting with $p = 7$ in Figure \ref{fig:7_7_caterpillar}. What we notice is that in level $\Gamma_0(7)$, the four spikes in the normalized slope multiplicity occur with period $21 = 7\cdot 3$ slopes. Moving to Figure \ref{fig:5_5_caterpillar}, the spikes occur with period $10 = 5 \cdot 2$, and a particularly large spike occurs just past $50 = 5^2 \cdot 2$.

 Based on this, we guess that multiplicity spikes occur periodically with period $p\cdot\frac{p-1}{2}$. Within that period behavior, the spikes are even taller every $p^2\cdot \frac{p-1}{2}$, and even taller than that every $p^3\cdot \frac{p-1}{2}$, and so on. For $p = 3$, we should see spikes every $3$ slopes, with larger spikes every $9$, extra large spikes every $27$ and extra super duper large spikes every $81$. While the every-$3$-slopes spike is difficult to see in Figure \ref{fig:3_3_caterpillar}, the other spike predictions essentially check out.
 
This begs the question:\ {\em why} are these spikes in slope multiplicity presenting themselves? At present, we have no answer to this question. The spikes themselves present a challenge to Conjecture \ref{conj}, since that conjecture predicts there is no particular bias towards one range of slopes over the other, as $k \rightarrow \infty$.

\subsection{$p$-adic slopes of algebraic central $L$-values and central $p$-adic $L$-derivatives}

Recall that we computed $\L$-invariants as a ratio of a derivative of a $p$-adic $L$-function and an $L$-value. Since Conjecture \ref{conj} proposes a statistical distribution of $\L$-invariants, it is natural to ask also about the distribution of the valuations of the numerator and denominator in such ratios.

Figure \ref{fig:5_5_Lderiv} presents three scatter plots. The top left is a smaller version of Figure \ref{fig:5_5_raw}, in which we plot slopes of $\L$-invariants versus weights, in level $\Gamma_0(5)$. The top right portion of the figure plots the valuations of $L_\infty(f,k/2)$, while the bottom portion of the figure plots the valuations of $L_p'(f,k/2)$. (At least, with the caveat that we have to twist $f$ by a Dirichlet character.)

\begin{figure}[htbp]
\centering
\includegraphics{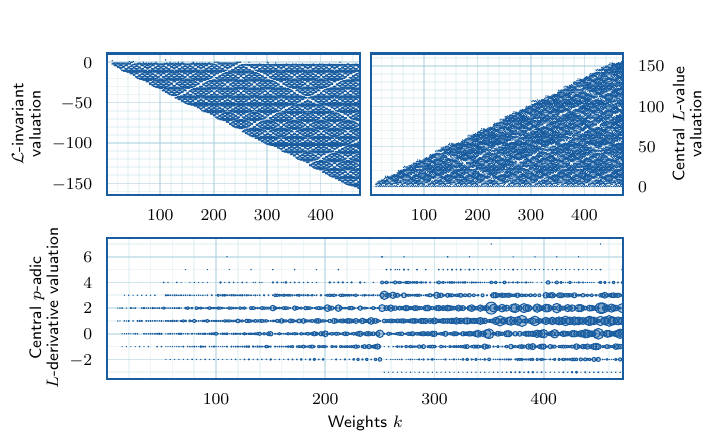}
\caption{Multiplicities of the valuations of (i) $\L$-invariants (top left) versus (ii) algebraic central $L$-values (top right) versus (iii) $p$-adic central $L$-derivatives (bottom) for newforms of $\Gamma_0(5)$ and weight $10 \leq k \leq 268$. (Total data:\ 5980 eigenforms)}
\label{fig:5_5_Lderiv}
\end{figure}

Note that the magnitude of $v_p(L_p'(f,k/2))$ is overall much smaller than the other data. In fact, the maximum datum is slope 7 and the minimum is slope -3. Therefore, we find the distribution of the valuations of $L_\infty(f,k/2)$ looks like the negative of the distribution of $\L$-invariants.  In particular, these $L$-values become very divisible by $p$ when the corresponding $\L$-invariant has very negative valuation.  What could be causing this behavior?

First, note that $\L$-invariants are local invariants in the Galois-theoretic perspective. So their sheer negativity is somehow a local phenomenon. From this, we reasonably guess that the values $L_\infty(f,k/2)$ that are highly divisible by $p$ are explained by local phenomena as well. From the perspective of the Birch and Swinnerton-Dyer conjecture, the $p$-adic local part of the algebraic central value $L_\infty(f,k/2)$ is the Tamagawa factor at $p$. Perhaps that factor is becoming divisible by $p$ in a systematic fashion over the data sets we gathered.
 
It would be interesting to explore this heuristic for modular eigenforms of any weight. Here we can at least look at the simplest case of an eigenform attached to an elliptic curve. The elliptic curve $E_{/\Q}$ in question would have split multiplicative reduction at $p$ and we write $q_E$ for its Tate parameter at $p$. Then, the Tamagawa factor at $p$ is given by $\ord_p(q_E)$, while the $\L$-invariant $E$ is given by $\log_p(q_E)/\ord_p(q_E)$. So, there is a clear connection between the $p$-divisibility of the Tamagawa factor and $p$-adic denominators of  $\L$-invariants. You can even see in this discussion that you would expect the valuation of $L_p'(E,k/2)$ to have a slight bias towards being positive, accounting for the fact that $\log_p(q_E)$ must be divisible by $p$. That bias is present in the bottom plot of Figure \ref{fig:5_5_Lderiv}.

\subsection{On the integrality of slopes of $\L$-invariants}

We end this paper with a question. To set the context, let $f$ be an eigenform on $\Gamma_0(N)$ with $p \nmid N$. Assume as well that $f$ is regular, which we recall means that $\rbar_f = \rhobar_f|_{\Gal(\Qpbar/\Qp)}$ is reducible. It is a folklore conjecture that $v_p(a_p(f))$ must be an integer. 

In fact, this global phenomenon is a consequence of a local conjecture. Namely, assume $v_p(a) > 0$ and let $r_{k,a}$ denote the unique irreducible crystalline representation with Hodge-Tate weights $0$ and $k-1$ on which Frobenius has characteristic polynomial $x^2 - a x + p^{k-1}$.  The local conjecture proposes that when $\overline{r}_{k,a}$ is regular and $k$ is even, then $v_p(a)$ is integral. See \cite[Conjecture  4.1.1]{BG-Survey}. Under relatively mild hypotheses, this local conjecture (and thus the global one) has been proven in \cite[Theorem 1.9]{LTXZ-Proof} using the machinery of the ghost conjecture.

We now consider the local conjecture from the perspective of Section \ref{subsec:rbar-loci}.  Fix $k$ even and $\rbar$ a regular local representation. Define $X_k^{\mathrm{crys}}(\rbar)$ within the open disc $v_p(a) > 0$ as those $a$ for which $\rbar \simeq \rbar_{k,a}$. This is a standard subset of the open disc, and the local conjecture implies that the valuation map $v_p : X_k^{\mathrm{crys}}(\rbar) \rightarrow \mathbb{Q}_{>0}$ takes values in the positive whole numbers. In particular, $v_p(-)$ is constant on connected components of $X_k^{\mathrm{crys}}(\rbar)$. Of course, $v_p(-)$ would {\em not} be constant on any component containing $a = 0$, and as a soft consistency check we know that $\rbar_{k,a}|_{a=0} \simeq \ind(\omega_2^{k-1})$ is always irregular.

What happens for $\L$-invariants? What is the image of the valuation map $v_p : X_k^{\pm}(\rbar) \rightarrow \mathbb P^1(\Qpbar)$? We only have global data through which to study this question. In looking at our data, we found that if $f$ was regular, then we had $v_p(\L_f) \in \mathbb Z$ the vast majority of the time. But, there were exceptions. Some examples occur for $3$-adic $\L$-invariants of eigenforms of level $\Gamma_0(21)$ that are globally lifts of $1 \oplus \omega$. Yet others occurred for $5$-adic $\L$-invariants lifting globally irreducible $\rhobar$ which are nevertheless locally split at $p$.

The examples we found illustrate one difference between slopes of $a_p$ and slopes of $\L$-invariants.  The issue is that the representation $\rbar_{k,\L}^{\pm}|_{\L=0}$ can be either regular or not, depending on the weight $k$. See, for instance, the computations of Breuil and M\'ezard for $2 < k < p$ and $k$ even \cite[Th\'eor\`eme 4.2.4.7]{BM}. And, naturally, when $\rbar = \rbar_{k,\L}^{\pm}|_{\L=0}$ is regular, there is some component of $X_k^{\pm}(\rbar)$ over which $v_p(\L)$ will vary. We suspect this explains why $v_p(\L_f)$ is sometimes non-integral, even when $\rbar_f$ is regular. Indeed, each time we found $v_p(\L_f) \not\in \mathbb Z$ but $\rbar_f$ was regular, it also happened that $v_p(\L_f)$ was the largest ({\it i.e.}\ most positive!) datum among all the $\rhobar_f$-data gathered in the same weight as the weight of $f$.

To end this article, we ask a precise question related to integrality properties for $v_p(\L_f)$ when $f$ is a regular eigenform.

\begin{question}\label{question:final}
Fix an even integer $k \geq 2$, a choice of sign $\pm$, and $\rbar: \mathrm{Gal}(\Qpbar/\Qp) \rightarrow \GL_2(\Fpbar)$. Suppose that $\rbar$ is regular and  $X \subseteq X_k^{\pm}(\rbar)$ is a connected component that does not contain $\L = 0$.
\begin{itemize}
\item Is it true that $v_p : X \rightarrow \mathbb P^1(\Qpbar)$ is constant?
\item If it is constant, is true that $v_p(X)$ is a single integer?
\end{itemize}
\end{question}

Note that there is another obvious way $X$ in Question \ref{question:final} would be forced to have non-constant valuations, because {\em a priori} it might be that $X$ is a connected component around $\infty \in \mathbb P^1(\Qpbar)$. However, that case is ruled out if $\rbar$ is regular, in the same way that $a = 0$ is ruled out in the crystalline case:\ the representation $\rbar_{k,\L}^{\pm}|_{\L=\infty} \simeq \ind(\omega_2^{k-1})$ is irreducible, hence irregular. See the discussion in Section \ref{subsec:thresholds-galois}.

We know of no progress towards such a local statement, but there is progress towards an integral $\L$-invariant slope conjecture for modular forms (with some controlled exceptions).  Namely, Jiawei An, assuming the ghost conjecture, has proven that $\L$-invariants of regular eigenforms in even weight $k$ have integral slope with at most $O(\log k)$ exceptions \cite{An}.

\bibliography{../Linv-bib.bib}
\bibliographystyle{abbrv}

\end{document}